\documentclass[12pt,a4paper,oneside]{article}
\usepackage{arxiv}

\usepackage[utf8]{inputenc}
\usepackage[numbers]{natbib}
\usepackage{graphicx}
\usepackage{placeins} 

\usepackage{url}

\usepackage{amsmath}
\usepackage{amsfonts}       
\usepackage{enumitem}
\usepackage{tabularx,booktabs} 
\usepackage{makecell}
\usepackage{multirow, makecell}

\usepackage{pdflscape}

\usepackage{wrapfig} 
\usepackage{empheq} 

\usepackage{amsthm}
\newtheorem{theorem}{Theorem}[section]
\newtheorem{corollary}[theorem]{Corollary}
\newtheorem{lemma}[theorem]{Lemma}
\newtheorem{proposition}[theorem]{Proposition}
\newtheorem{remark}[theorem]{Remark}
\newtheorem{definition}{Definition}
\newtheorem{assumption}{Assumption}

\newtheorem*{notation}{Notation}

\usepackage{ulem}
\usepackage{mathtools}

\usepackage[x11names]{xcolor}

\usepackage[colorlinks=true, linkcolor=blue, citecolor=blue, urlcolor=blue]{hyperref}
\usepackage{float}
\newcommand{\R}{\mathbb{R}}
\newcommand{\N}{\mathbb{N}}

\newcommand{\norm}[1]{\left\| #1 \right\|}
\newcommand{\prox}[1]{\operatorname{prox}_{#1}} 
\newcommand{\Id}{\operatorname{Id}}
\newcommand{\lc}{L_{\varphi}}
\newcommand{\lp}{L_{\Phi}}
\usepackage{xspace} 
\newcommand{\algoone}{\textbf{GD-RGA}\xspace}
\newcommand{\algotwo}{\textbf{PD-RGA}\xspace}
\usepackage{comment}
\numberwithin{equation}{section}
\newcommand{\gt}[1]{{\color{black}#1}}
\newcommand{\ifct}{\iota}

\title{On the Convergence of Proximal Algorithms for Weakly-convex Min-max Optimization}

\author{
	\textbf{Guido Tapia-Riera}$^\ast$ \quad \textbf{Camille Castera} \quad \textbf{Nicolas Papadakis}\\
	Univ. Bordeaux\\
	CNRS, Bordeaux INP, IMB, UMR 5251\\
	F-33400 Talence, France
}

\makeatletter

\begin{document}

	\maketitle

	\begin{abstract}
	We study alternating first-order algorithms with no inner loops for solving nonconvex--strongly-concave min-max problems. We show the convergence of \gt{the alternating gradient descent--ascent algorithm} by proposing a substantially simplified proof compared to previous ones. It allows us to enlarge the set of admissible step-sizes. Building on this general reformulation, we also prove the convergence of a doubly proximal algorithm in the weakly convex--strongly concave setting. Finally, we show how this new result opens the way to new applications of min-max optimization algorithms for solving regularized imaging inverse problems with neural networks in a plug-and-play manner.
	\end{abstract}
	
	\renewcommand*{\thefootnote}{$^\ast$}
	\footnotetext[1]{Corresponding author: \texttt{guido-samuel.tapia-riera@math.u-bordeaux.fr}}
	\renewcommand*{\thefootnote}{\arabic{footnote}}
	\setcounter{footnote}{0} 

\section{Introduction}
	A challenge in min-max optimization is to design algorithms that remain applicable in the absence of convexity and/or concavity of the objective functions. This challenge arises in tasks such as training generative adversarial networks (GANs) \cite{goodfellow2020GANS}, adversarial learning \cite{pmlr-v119-liu20j}, robust training of neural networks \cite{madry2018towards}, distributed signal processing \cite{Giannakis2016,8340193}, (\gt{see \cite{razaviyayn2020nonconvex} for further details on applications}). We consider optimization problems of the form
	\begin{equation}\label{introduction_eq_1}
		\min_{x \in \R^d} \max_{y \in \R^n}  \left\{ \Phi(x,y)-h(y) \right\}
	\end{equation}
	where $\Phi: \R^d \times \R^n \to \R$, called the \emph{coupling}, 
	is differentiable with Lipschitz-continuous gradient. We assume that $\Phi$ is concave with respect to (w.r.t.) $y$, and $\rho-$weakly convex in $x$, for some $\rho>0$. This means that $\forall y\in\R^n$, $\Phi(\cdot, y)$ is possibly nonconvex but \gt{$\Phi(\cdot,y) + \rho/2\norm{\cdot}^2$ is convex}. The function $h: \R^n \to \R \cup \{+\infty\}$ is proper, convex and lower semicontinuous, and referred to as the \emph{regularizer}.
	
	This paper studies the following two alternating algorithms for tackling \eqref{introduction_eq_1}: for all $k \geq 0$,\\
	
	\begin{minipage}[t]{0.47\textwidth}
		{\footnotesize Gradient Descent--Regularized Gradient Ascent}\\
		(\algoone) 
		$$
		\begin{cases}
			x_{k+1} = x_k - \eta_x \nabla_x \Phi(x_k,y_k) \\
			y_{k+1} = \prox{\eta_y h}\left(y_k + \eta_y \nabla_y \Phi(x_{k+1}, y_k)\right),
		\end{cases}	
		$$
	\end{minipage}
	\hspace{0.4cm}
	\begin{minipage}[t]{0.48\textwidth}
		{\footnotesize Proximal Descent--Regularized Gradient Ascent}\\
		(\algotwo)
		$$
		\begin{cases}
			x_{k+1} = \prox{\eta_x \Phi(\cdot,y_{k})}(x_k)\\
			y_{k+1} = \prox{\eta_y h}\left(y_k + \eta_y \nabla_y \Phi(x_{k+1}, y_k)\right),
		\end{cases}
		$$
	\end{minipage}
	\vspace{0.2cm}\\
	where $\eta_x, \eta_y >0$ are step-sizes, $\prox{}$ is the proximal operator (see Definition~\ref{proximal_operator}), and $\nabla_x, \nabla_y$ denote the partial derivative (or gradient) w.r.t. $x$ and $y$, respectively. 
	
	The main advantage of \algoone and \algotwo is that they do not require solving exactly the inner problem (the $\max$ in \eqref{introduction_eq_1}) at each iteration. As such they circumvent the computational expensiveness of bi-level optimization approaches that make use of inner loops \cite{ijcai2025p741-bilevel,bard1998practical,dempe2002foundations}. It is well known that when the coupling term $\Phi$ is bilinear, the iterates of simultaneous (\emph{i.e.,} both variables update based on the previous iterate) first-order methods used to solve problem~\eqref{introduction_eq_1} may diverge, whereas those of the alternating scheme remain bounded. In recent years, this behavior, together with more favorable convergence properties, has led alternating methods to be preferred over their simultaneous counterparts in practice~\cite{pmlr-v125-bailey20a,pmlr-v89-gidel19a,Xu2023projalgo}. \algoone relies on explicit gradient descent--ascent updates of the coupling term $\Phi$
	and has previously been studied \cite{bot2023lternatingproximalgradientstepsstochastic,lin2025two}. 
	
	In contrast, proximal (or implicit) updates of $\Phi$, as in \algotwo, have received significantly less attention. Our interest in this second approach is motivated by recent advances 
	where neural networks are designed and trained to act as firmly non-expansive operators~\cite{pesquet2021learning}, as proximal operators of weakly convex functions~\cite{Hurault2022ProximalDF}, \gt{or by extending input convex neural networks of~\cite{amos2017input} to the weakly convex setting, resulting in \emph{input weakly convex neural networks} (IWCNNs)~\cite{shumaylov2024weakly,shumaylov2023provably}, which, as their name suggests, are weakly convex functions of their input}.
	
	Algorithms for solving \eqref{introduction_eq_1} have been thoroughly studied when $\Phi$ is \emph{convex--concave} \cite{chambolle2011first,pmlr-v125-golowich20a,Hamedani2021primaldual,Daskalakis2021Thecomplexityofconstrainedminmaxoptimization}. There, under appropriate constraint qualifications \cite{rockafellar-1970a,correa2023fundamentals}, algorithms may converge to so-called saddle points $(x^\ast,y^\ast)$, \textit{i.e.,} satisfying $\Phi(x^*,y) - h(y) \le \Phi(x^*,y^*) - h(y^*) \le \Phi(x,y^*) - h(y^*)$, for all $(x,y) \in \mathbb{R}^d \times \mathbb{R}^n$. However, this characterization of solutions fails in the \emph{nonconvex--strongly concave} setting (only the upper inequality holds true). In prior work on the nonconvex setting \cite{bot2023lternatingproximalgradientstepsstochastic,Cohen2025AlternatingandParallelProximalGradientMethods,lin2025two}, the problem~\eqref{introduction_eq_1}  is reformulated as
	\begin{equation}\label{introduction_eq_2}
		\min_{x \in \mathbb{R}^d} \; \varphi(x),\quad \text{where}\quad \varphi(x):= \max_{y \in \mathbb{R}^n} \; \Phi(x,y) - h(y),
	\end{equation}
	and the notion of solution is relaxed to merely finding approximations of critical points of $\varphi$. Our work considers the case where \eqref{introduction_eq_2} is nonconvex but differentiable, for which we will require either \gt{$-h$ to be strongly concave} or $\Phi$ to be strongly concave w.r.t.\ its second argument (see Proposition~\ref{Smoothness_of_the_max_function}). In this setting, we aim to find $\varepsilon$-stationary points, defined next.
	
	\begin{definition}[$\varepsilon$-stationary points]\label{def_epsilon_stationary} 
		Let $\varepsilon \geq 0$. We say that $x$ is a $\varepsilon$-stationary point of $\varphi$ if $\norm{\nabla \varphi(x)} \leq \varepsilon$. If $\varepsilon=0$, then $x$ is a stationary point.
	\end{definition}
	
	\subsection{Main contributions}
	For \algoone, we revisit the proof of convergence to the $\varepsilon$-stationary points of \eqref{introduction_eq_2}.
	By getting rid of approximation in the key parts of the analysis (see Lemma~\ref{Algo1-lemma-3}), we obtain tighter bounds on the range of admissible step-sizes for \algoone (see Theorem~\ref{thm_non_convex_strongly_concave_case_1}) and provide a significantly shorter proof than that of prior work \cite{bot2023lternatingproximalgradientstepsstochastic,lin2025two}.
	We prove the convergence of \algotwo (Theorem~\ref{main_thm_algo_2}) with the same complexity as \algoone. To the best of our knowledge, this is the first proof of convergence for this algorithm in the weakly-convex setting. Finally, we use \algotwo to tackle imaging inverse problems where the proximal operator $\prox{\eta_x \Phi(\cdot,y_k)}$ in \algotwo is implemented by a neural network in a plug-and-play fashion.

	\subsection{Related Work}
	Lower bounds have been derived for first-order algorithms in the nonconvex-strongly-concave smooth setting~\cite{pmlr-v161-zhang21c,NEURIPS2021_0e105949}. They state that a dependence of order $\varepsilon^{-2}$ in the number of iterations required to reach an $\varepsilon$-stationary point is unavoidable, and our methods attain this rate. We point out that stochastic first-order methods for solving~\eqref{introduction_eq_1} have also been studied, primarily using the geometrical structure of the objective function, in particular the Polyak--\L{}ojasiewicz condition~\cite{Liu2020Stochastic}, and variance reduction techniques~\cite{rafique2022weakly}. While nonsmooth first-order methods arise when, instead of assuming strong concavity, one considers concavity or nonconcavity in the variable $y$ of the coupling term $\Phi$ or in the regularizer $-h$ or when non-smoothness is assumed in the coupling term~\cite{liu2021first,razaviyayn2020nonconvex,lin2025two}. 
	
	\textbf{Closely-related work on \algoone.} \ \algoone can be seen as a generalization, in an alternating way, of the smooth version of the algorithm proposed in~\cite{lin2025two}, where convergence to $\varepsilon$-stationary points is established under the assumption that $\Phi(x,\cdot)$ is strongly concave for all $x$ and that $h$ is the indicator function of a nonempty, convex, and bounded set for which $\prox{h}$ is the projection on this set. We relax the strong concavity assumption on $\Phi(x,\cdot)$ by allowing it to be \emph{optionally} carried by $-h$. In addition, our analysis allows for general $h$ and corresponding proximal operators $\prox{h}$. In particular, when $\Phi(x,\cdot)$ is strongly concave we drop the regularization by taking $h \equiv 0$ which is not possible in \cite{lin2025two}.  Finally,  we provide a larger range of admissible step-sizes than in~\cite{lin2025two}.
	
	The algorithm \algoone has also a strong connection with Algorithm~2.1 in~\cite{bot2023lternatingproximalgradientstepsstochastic} where the authors include an additional convex regularizer $f$ 
	w.r.t. to the variable $x$, and used via its proximal operator $\prox{f}$. The results therein apply to the setting we consider for \algoone when the strong concavity originates from $-h$. The key difference in their analysis is that they use the properties of the proximal operator associated with $h$, $\prox{h}$, when $\Phi$ is strongly concave w.r.t. $y$, we depart from this approach and rather make use of the ascent lemma (see Lemma~\ref{Descen_lemma}). This separation allows us to relax the conditions imposed on the step-sizes.
	
	We note that both aforementioned works assume that the coupling $\Phi$ has a jointly Lipschitz continuous gradient, \textit{i.e}, there exists $L_{\Phi}>0$ such that $\norm{\nabla \Phi(x,y) - \nabla \Phi(x',y')} \leq L_{\Phi} \norm{(x,y)-(x',y')}, \quad \forall x,x' \in \R^d,\ y,y' \in \R^n.$ We relax this assumption by requiring only block-wise Lipschitz continuity of the gradient (see Assumption \ref{assumption_1}, \ref{assumption_2}, and \ref{assumption_3}), which allows us to get a refined range of step-sizes, by considering tighter inequalities.
	
	Similarly to~\cite{bot2023lternatingproximalgradientstepsstochastic}, the work~\cite{Cohen2025AlternatingandParallelProximalGradientMethods} introduces a regularizer $f$ w.r.t.\ the $x$-variable. Rather than focusing on $\varepsilon$-stationary points, they establish that any accumulation point of $\varphi + f$ is indeed a critical point under additional assumptions than those considered in our work, in particular the semialgebraicity of the objective function, which are beyond the scope of our work. Since we share the block-wise Lipschitz continuity assumption of the gradient of $\Phi$, the mapping $\varphi$ exhibits properties similar to those studied in~\cite{bot2023lternatingproximalgradientstepsstochastic}. Furthermore, when $f \equiv 0$, their algorithm reduces to \algoone; however, the corresponding convergence analyses differ significantly.
	
	\gt{
		We point out that other methods \cite{Zhang2020Asingleloopsmothed, li2026smoothing,Xu2023projalgo} tackle weakly-convex concave min-max problems but only over closed convex domains $\mathcal{X}\times\mathcal{Y}\subset \R^d \times \R^n$ and  require $\mathcal{Y}$ to be bounded. Those methods are built upon Alternating Gradient Projections (AGP)~\cite[Sect. 3.1]{Xu2023projalgo}. Their analysis carry through the concave case by using asymptotically vanishing quadratic regularizations \cite{Xu2023projalgo} or a Moreau-Yoshida one \cite{Zhang2020Asingleloopsmothed}. The two are combined \cite{li2026smoothing} to get improved complexity results.
		The key difference is that our approach allows considering $\mathcal{Y}$ unbounded. We can compare the approaches when $\Phi$ is strongly concave w.r.t. $y$. Indeed, in that case, AGP can be seen as a special case of \algoone where $h$ is the indicator function of $\mathcal{Y}$. 
		The aforementioned algorithms and \algoone thus  follow two different approaches that are comparable only in this specific case, where they both yield similar $\varepsilon^{-2}$ complexity results.
	}

	\textbf{Related work for \algotwo.} \ While \algotwo is less studied, \cite{bot2022acceleratedminimaxalgorithmconvexconcave} considers a similar algorithm that differs by featuring an additional inertial step on the variable $y$. 
	However the theoretical analysis in~\cite{bot2022acceleratedminimaxalgorithmconvexconcave} covers the convex (and non-smooth) setting, and is only valid for non-zero momentum parameters. In comparison, we do not consider inertial steps, and allow the function to be weakly-convex in $\Phi$ w.r.t. $x$, which requires a significantly different analysis.
	
	\textbf{Organization.} \ The paper is organized as follows. Section~\ref{section_preliminary} provides tools of convex and nonconvex optimization including regularity properties of the mapping $\varphi$ defined in \eqref{introduction_eq_2}. Convergence of \algoone and \algotwo to $\varepsilon$-stationary points is studied in Sections~\ref{section_convergence_analysis_for_algorithm_1} and~\ref{section_convergence_analysis_for_algorithm_2}, respectively.  Finally, Section~\ref{section_numerical_experiments} presents numerical experiments, including Plug-and-Play image restoration. 

	
	\section{Preliminary}\label{section_preliminary}
	
	In this section, we provide the necessary background on weakly and strongly convex functions, smooth (differentiable) functions, proximal operators associated with weakly convex functions, and regularity properties of the mapping $\varphi$ defined in~\eqref{introduction_eq_2}. These concepts constitute essential tools for the convergence analysis of \algoone and \algotwo towards an $\varepsilon$-stationary point.
	
	\begin{notation} 
		We denote by $\|\cdot\|$ the $\ell^2$-norm, by $\Id$ the identity operator, by $\mathcal{O}$ the big O notation, by $\N_0 = \N \cup \{0\}$, by $\norm{\cdot}_S$ the spectral norm of a matrix, and \gt{by $\ifct$ the set indicator function}.
	\end{notation}
	
	\subsection{Tools for min-max optimization}
	\begin{definition}[Strong/weak convexity]\label{def_strong_weak_convexity} 
		Let $g: \R^d \to \R \cup \{+\infty\}$ and $\alpha>0$, then $g$ is $\alpha$-strongly convex if $ g(x)- \frac{\alpha}{2} \norm{x}^2$ is convex and  $\alpha$-weakly convex if $ g(x)+ \frac{\alpha}{2} \norm{x}^2$ is convex.
	\end{definition}
	We say that a function $f$ is strongly concave if $-f$ is strongly convex.
	
	\begin{definition}[Subdifferential]\label{subdiferential_definition}
		Let $g : \R^d \to \R \cup \{+\infty\}$ and $x, z \in \mathbb{R}^d$.
		If for all $x' \in \mathbb{R}^d$ we have that 
		\begin{equation*}
			g(x') \geq g(x) + \langle z, x' - x \rangle,
		\end{equation*}
		we say that $z$ is a subgradient of $g$ at $x$. The set of all subgradients of $g$ at $x$, denoted by $\partial g(x)$, is called the subdifferential of $g$ at $x$.
	\end{definition}
	
	\begin{definition}[$L$-smoothness]
		We say that a differentiable function $g$ is $L$-smooth if $\nabla g$ is $L$-Lipschitz continuous, that is, there exists $L\geq 0$ such that for all $x,x' \in\R^d$,
		$$
		\norm{\nabla g(x) - \nabla g(x')} \leq L \norm{x - x'}.
		$$
	\end{definition}
	
	We recall an important property for $L$-smooth functions, called the descent lemma, see e.g. \cite[Lem.1.2.3, pp. 25]{Nesterovbookconvexopti}.
	\begin{lemma}[Descent/Ascent lemma]\label{Descen_lemma}
		For an $L$-smooth function $g : \mathbb{R}^d \to \mathbb{R}$ we have that
		\begin{equation*}
			\lvert g(x') - g(x) - \langle \nabla g(x), x' - x \rangle \rvert \leq  \frac{L}{2} \|x' - x\|^2 \quad \forall x, x' \in \mathbb{R}^d.
		\end{equation*}
	\end{lemma}

	We now define the proximal operator, at the heart of \algoone and \algotwo.
	
	\begin{definition}[Proximal operator]\label{proximal_operator}
		Let $g: \R^d \to \R$ be $\rho$-weakly convex, proper and lower \gt{semicontinuous}  (l.s.c) and let $\tau>0$. For all $x\in\R^d$, the proximal operator of $g$ at $x$ with step-size $\tau$ is defined as:
		\begin{equation*}\label{proximal_operator_definition}
			\displaystyle 
			\prox{\tau g} (x) \in \operatorname{arg min}_{z \in \R^d} \left\{ g(z) + \frac{1}{2 \tau}\norm{z-x}^2 \right\}.
		\end{equation*}
	\end{definition}
	The set $\prox{\tau g} (x)$ is non-empty, single-valued when $\tau\rho<1$ \cite[Def. 2.2]{renaud2025moreauenvelope}.
	Convex functions are $0$-weakly convex, making their proximal operator always single-valued\footnote{
		\gt{See \textit{,e.g.,} \url{http://proximity-operator.net/} for further details on proximal operators and implementation.}}.

	\subsection{Ingredients for the convergence analysis}
	In this subsection, we present important properties of the mapping $\varphi$. To this end, we make the following assumptions:
	\begin{assumption}\label{assumption_1}
		\begin{enumerate}[label=(\roman*)]
			\item Let $y \in \R^n$ there exists $L_{yx}>0$ such that for all $x_1,x_2 \in \R^d$ we have that $\norm{\nabla_y \Phi(x_1,y)-\nabla_y \Phi(x_2,y)} \leq L_{yx} \norm{x_1-x_2}$.
			\item For all $x \in \R^d$, $\Phi(x,\cdot)$ is concave. Furthermore, there exists $\mu>0$ such that either $\Phi(x,\cdot)$ is $\mu$-strongly concave for all $ x \in \R^d$, or $-h$ is $\mu$-strongly concave.
			\item The regularizer $h$ is proper, l.s.c, and convex. 
		\end{enumerate}
	\end{assumption}
	
	Now, thanks to Assumption~\ref{assumption_1}, we can establish the well-posedness of $\varphi$ using \cite[Lem. 1, pp. 146]{Cohen2025AlternatingandParallelProximalGradientMethods}:
	
	\begin{lemma}[Lipschitz continuity of the solution mapping]\label{Lipschitz_continuity_of_the_solution_mapping}
		Under Assumption \ref{assumption_1} the solution map 
		$\R^d \ni
		x \longmapsto y^*(x)=\operatorname{argmax}_{y \in \R^n} \bigl\{ \Phi(x,y) - h(y) \bigr\} \in \R^n$
		is single-valued and it is $L_{yx}/\mu$-Lipschitz.
	\end{lemma}

	Note that $\forall x \in \mathbb{R}^d$, $\varphi(x) = \Phi(x, y^*(x))-h(y^*(x)) = \max_{y \in \R^n} \left\{ \Phi(x, y)-h(y) \right\}$ thanks to the definition of $y^*$. Now, we establish that the mapping $\varphi$ is continuously differentiable ($C^1$). 
	To this end we make the next assumption:
	
	\begin{assumption}\label{assumption_2}
		\begin{enumerate}[label=(\roman*)]
			\item Let $x \in \R^d$ and $y \in \R^n$ there exist $L_{xx}, L_{xy}>0$ such that for all $x_1,x_2 \in \R^d$ and $y_1,y_2 \in \R^n$ we have
			\begin{enumerate}
				\item  $\norm{\nabla_x \Phi(x_1,y)-\nabla_x \Phi(x_2,y)} \leq L_{xx} \norm{x_1-x_2}$, 
				\item $\norm{\nabla_x \Phi(x,y_1)-\nabla_x \Phi(x,y_2)} \leq L_{xy} \norm{y_1-y_2}$.
			\end{enumerate}
			\item For all $y \in \R^n$, $\Phi(\cdot,y)$ is $\rho$-weakly convex.
			\item The function $\varphi$ defined in \eqref{introduction_eq_2} is lower bounded, i.e., $\inf_{x \in \mathbb{R}^d} \varphi(x) > -\infty$.
		\end{enumerate}
	\end{assumption}
	
	\begin{proposition}[Smoothness of $\varphi$]\label{Smoothness_of_the_max_function}
		Under assumptions \ref{assumption_1} and \ref{assumption_2} we have that $\varphi$ is $\rho-$weakly convex and of class $C^1$ with gradient given by $\nabla \varphi(x) = \nabla_x \Phi(x, y^*(x))$.
		Furthermore, its gradient is $L_{\varphi}:=L_{xx} + (L_{xy}L_{yx})/\mu$-Lipschitz continuous.
	\end{proposition}
	\begin{proof}
		The regularity part ($C^1$) of this result follows from \cite[Prop. 1, pp. 146]{Cohen2025AlternatingandParallelProximalGradientMethods}  while the Lipschitzness of $\nabla \varphi$ is given by \cite[Lem. 2, pp. 147]{Cohen2025AlternatingandParallelProximalGradientMethods}. The $\rho-$weakly convexity follows directly using Assumption~\ref{assumption_2}-$(ii)$ and from the fact that supremum preserves the triangular inequality.
	\end{proof}
	
	To end this section, let us introduce a sufficient condition for a sequence $(w_k)_{k \geq 0}$ to achieve an $\varepsilon$-stationary point (Definition~\ref{def_epsilon_stationary}) for the mapping  $\varphi$ defined in~\eqref{introduction_eq_2}.
	
	\begin{proposition}\label{sufficient_condition_epsilon_stationary_point}
		Let $N \geq 1$, $(w_k)_{k=0}^N \subset \R^d$, and \gt{Assumption~\ref{assumption_2}-(iii)} hold true. Assume that the mapping $\varphi$ satisfies
		\begin{equation}\label{final_prop_intro_eq_1}
			\varphi(w_N) \leq \varphi(w_0) - C_1 \sum_{k=0}^{N-1}\norm{\nabla \varphi(w_k)}^2 + C_2  \quad \text{where} \quad C_1, C_2 \in \R.
		\end{equation}
		If $C_1,C_2 >0$, then for any $\varepsilon>0$ there exists $k=\mathcal{O}(\varepsilon^{-2})$ such that $\norm{\nabla \varphi(w_k)}<\varepsilon$.
	\end{proposition}
	\begin{proof}
		By \gt{Assumption~\ref{assumption_2}-(iii)} we have that $\varphi(w_0) - \varphi(w_N) < \varphi(w_0) - \inf_w \varphi(w) < +\infty$, \emph{i.e} there exists $C_3 \in \R$ such that $\varphi(w_0) - \inf_w \varphi(w) < C_3$. Combining this fact with~\eqref{final_prop_intro_eq_1}, we deduce that
		\begin{equation*}
			\min_{0 \leq k \leq N-1} \norm{\nabla \varphi(w_k)}^2 \leq  \frac{1}{N}\sum_{k=0}^{N-1}\norm{\nabla \varphi(w_k)}^2 < \frac{C}{N},
		\end{equation*}
		where $C = (C_3 + C_2)/C_1>0$. Thus for any $N$ there exists $ 0 \leq k < N$ such that $\norm{\nabla \varphi(w_{k})} < \sqrt{C/N}$. Given $\varepsilon>0$ we choose $N = \lceil C/\varepsilon^2 \rceil$ so $k\leq N = \mathcal{O}(\varepsilon^{-2}$) and we have that $\norm{\nabla \varphi(w_{k})} < \varepsilon$.
	\end{proof}
	

	
	\section{Convergence analysis for \algoone}\label{section_convergence_analysis_for_algorithm_1}
	We recall that \algoone is defined $\forall k\geq 0$ by
	\begin{align}
		x_{k+1} &= x_k - \eta_x \nabla_x \Phi(x_k,y_k) \label{algo_1_update_x_sect_3}\\
		y_{k+1} &= \prox{\eta_y h}\left(y_k + \eta_y \nabla_y \Phi(x_{k+1}, y_k)\right). \label{algorithm_dproxa_b_sect_3}
	\end{align}
	In this section, we prove the convergence of \algoone towards an  $\varepsilon$-stationary  point (Definition~\ref{def_epsilon_stationary}) of the mapping $\varphi$ defined in~\eqref{introduction_eq_2}. To this end, we make the following assumption.
	
	\begin{assumption}\label{assumption_3}
		\begin{enumerate}[label=(\roman*)]
			\item Let $x \in \R^d$ there exists $L_{yy}>0$ such that for all $y_1,y_2 \in \R^n$ we have  
			\begin{equation}\label{lyy-lipschitz-constant}
				\norm{\nabla_y \Phi(x,y_1)-\nabla_y \Phi(x,y_2)} \leq L_{yy} \norm{y_1-y_2}.
			\end{equation}
			\item The condition number satisfies $\kappa_y := L_{yy} / \mu \geq 1$.
		\end{enumerate}
	\end{assumption}
	\begin{remark}
		Note that when the strong concavity originates from the coupling term $\Phi$, we have $\kappa_y \geq 1$. In contrast, when it originates from the regularizer $-h$, this is no longer necessarily true. However, it is always possible to choose $L_{yy}$ sufficiently large so as to ensure that $\kappa_y \geq 1$.
	\end{remark}
	
	Let us now state the main result of this section.
	
	\begin{theorem}\label{thm_non_convex_strongly_concave_case_1}
		Let Assumptions~\ref{assumption_1}, \ref{assumption_2}, and~\ref{assumption_3} hold true. For any initial point $(x_0,y_0) \in \mathbb{R}^d \times \mathbb{R}^n$ we consider the sequence $(x_k,y_k)_{k \geq 1}$ generated by \algoone with step-sizes 
		\gt{
			\begin{equation}\label{constraint_eta_x_algo1}
				0<\eta_y \leq \frac{1}{L_{yy}} \quad \text{and} \quad 0<\eta_x < \min \left( \frac{1}{\lc},  \frac{\eta_yL_{yy}\mu}{2 \sqrt{2}\kappa_yL_{xy}L_{yx}} \right).
			\end{equation}
		}
		Then, for every $\varepsilon >0$ there exists $k = \mathcal{O}  (\varepsilon^{-2})$ such that $\norm{\nabla \varphi(x_k)} < \varepsilon$.
	\end{theorem}
	
	From now on and throughout this work, the step-size $\eta_y$ is parameterized as $\eta_y = \tau / L_{yy}$ for some $\tau \in (0,1]$. Consequently, the sequence $(x_k, y_k)_{k \geq 1}$ generated by \algoone depends on $\tau$. However, without loss of generality, we simply write $(x_k, y_k)_{k \geq 1}$ to avoid explicitly including $\tau$ in the notation.

	To prove Theorem~\ref{thm_non_convex_strongly_concave_case_1}, we first provide key results that form the basis of our analysis of \algoone.  We recall that $\varphi(x) = \max_{y \in \R^n} \ \left( \Phi(x,y)-h(y) \right)$ and that under Assumption~\ref{assumption_1} we write $y^*(x) = \operatorname{argmax}_{y \in \R^n}\left( \Phi(x,y)-h(y)\right)$.
	
	\begin{lemma}\label{Algo1-lemma-1}
		Let Assumptions \ref{assumption_1} and \ref{assumption_2} hold true and $(x_k,y_k)_{k \geq 1}$ be a sequence generated by \algoone and $N \geq 1$. \gt{If $\eta_x \lc < 1$, then we have}
		\gt{
			\begin{equation}\label{Algo1-lemma-1_main_eq}
				\varphi(x_{N}) \leq \varphi(x_0) - \frac{\eta_x}{2}\sum_{k=0}^{N-1}\norm{\nabla \varphi(x_k)}^2 + \frac{\eta_x}{2}L_{xy}^2 \sum_{k=0}^{N-1} \norm{y^*(x_k)-y_k}^2. 
		\end{equation}}
	\end{lemma}

	The proof of this result is presented in Appendix~\ref{Proof_lemma_1_algo_1}. One of the main difficulties in the analysis of convergence for both algorithms is to control the gap $\norm{y^*(x_k)-y_k}^2$ (for all $k \in \N$) between the theoretical maximizer $y^*(x_k)$ of the inner problem at iteration $k$ in~\eqref{introduction_eq_1} and the current iterate\footnote{Since the proximal step to update variable $y$ is the same in both \algoone and \algotwo, the following result will be reused for the analysis of Algorithm 2.} $y_k$ given by~\eqref{algorithm_dproxa_b_sect_3}.
	
	\begin{notation}
		To ease the computations, throughout the paper, we denote
		$ y^*(x_k)$ by $y^*_k$,  $ \norm{y^*_k-y_k}^2$ by $\delta_k$, and $\beta = \mu /(L_{xy}L_{yx})$.
	\end{notation}
	
	\subsection{Control of \texorpdfstring{$\delta_k$}{delta k}}
	The aim of this section is to derive a bound for $\delta_k$. The key difference between our convergence analysis and those in~\cite{lin2025two,bot2023lternatingproximalgradientstepsstochastic} lies in the construction of this control. Our approach yields a less restrictive bound and allows us to explicitly characterize an admissible range for the step-size $\eta_x$, a step not employed in the aforementioned works.
	
	\begin{lemma}[General bound for $\delta_{k}$]\label{Algo1-lemma-2}
		Let Assumptions~\ref{assumption_1} and \ref{assumption_3} hold true and $(y_k)_{k \geq 1}$ be a sequence generated by ~\eqref{algorithm_dproxa_b_sect_3}. Then for every $k \in \N_0$ we have 
		\begin{equation}\label{result_of_lema_2_Algo_1}
			\delta_{k+1} \leq \left( 1 - \frac{\tau}{2\kappa_y} \right) \delta_k + \frac{\kappa_y}{\tau} \norm{ y_{k+1}^* - y_k^* }^2 \quad \forall \ 0 < \tau \leq 1.
		\end{equation}
	\end{lemma}
	
	We now refine it by using the update rule of $x_k$ in \algoone.
	
	\begin{lemma}[Bound of $\delta_{k}$ for \algoone]\label{Algo1-lemma-3}
		Let Assumptions~\ref{assumption_1}, \ref{assumption_2} and \ref{assumption_3} hold true and $(x_k,y_k)_{k \geq 1}$ be a sequence generated by \algoone. Then for every $k \in \N_0$  we have
		\begin{equation}\label{Algo1-lemma-3-main-eq}
			\delta_k \leq \gamma^k \delta_0 + \left(\gamma -1 +\frac{\tau}{2\kappa_y} \right)\frac{1}{L_{xy}^2} \sum_{j=0}^{k-1} \gamma^{k-1-j}\norm{\nabla \varphi(x_j)}^2 \quad \forall \ 0 < \tau \leq 1,
		\end{equation}
		where $\gamma = 1 - \tau/(2\kappa_y) + 2\kappa_y\eta_x^2/(\beta^2\tau )$. Furthermore, if $\eta_x< \beta\tau /(2\kappa_y)$ then $|\gamma|<1$.
	\end{lemma}
	\begin{proof}
		Exploiting the $L_{yx}/\mu$-Lipschitz continuity of $y^*$ with respect to $x$ (see Lemma \ref{Lipschitz_continuity_of_the_solution_mapping}) and the $x$-update~\eqref{algo_1_update_x_sect_3} of \algoone relation~\eqref{result_of_lema_2_Algo_1} becomes
		\begin{equation*}
			\delta_{k+1} \leq \left(1 -\frac{\tau}{2\kappa_y} \right)\delta_k + \frac{\kappa_yL_{yx}^2\eta_x^2}{\tau \mu^2}\norm{\nabla_x \Phi(x_k,y_k)}^2.
		\end{equation*}
		Using Cauchy-Schwarz inequality on $\norm{\nabla_x \Phi(x_k,y_k)}^2 = \norm{\nabla_x \Phi(x_k,y_k) + \nabla \varphi(x_k)-\nabla \varphi(x_k)}^2$ we get
		\begin{equation}\label{Algo1-lemma-3-eq-1}
			\delta_{k+1} \leq \left( 1 - \frac{\tau}{2\kappa_y} \right) \delta_k + \frac{2\kappa_y L_{yx}^2 \eta_x^2}{\tau\mu^2} \norm{ \nabla \varphi(x_k) - \nabla_x \Phi(x_k, y_k) }^2 + \frac{2\kappa_yL_{yx}^2\eta_x^2}{\tau\mu^2} \norm{ \nabla \varphi(x_k) }^2.
		\end{equation}
		By Proposition~\ref{Smoothness_of_the_max_function} and Assumption~\ref{assumption_2}(i)-(b), $\norm{ \nabla \varphi(x_k) - \nabla_x \Phi(x_k, y_k) }^2 \leq L_{xy}^2 \norm{y^*_k -y_k}^2 = L_{xy}^2 \delta_k$, and~\eqref{Algo1-lemma-3-eq-1} becomes
		\begin{equation}
			\delta_{k+1} \leq \left( 1 - \frac{\tau}{2\kappa_y} + \frac{2\kappa_yL_{yx}^2L_{xy}^2\eta_x^2}{\tau\mu^2} \right)\delta_k + \frac{2\kappa_y L_{yx}^2\eta_x^2}{\tau \mu^2}\norm{\nabla \varphi(x_k)}^2. \notag 
		\end{equation}
		Using $\beta=\mu/(L_{xy}L_{yx})$ the last inequality writes
		\begin{equation}\label{Algo1-lemma-3-eq-2}
			\delta_{k+1} \leq \left( 1 - \frac{\tau}{2\kappa_y} + \frac{2\kappa_y\eta_x^2}{\beta^2\tau }  \right) \delta_k + \frac{2\kappa_y\eta_x^2}{\beta^2\tau L_{xy}^2} \norm{\nabla \varphi(x_k)}^2. 
		\end{equation}
		Denoting $ \gamma=1 - \tau/2\kappa_y + 2\kappa_y\eta_x^2/(\beta^2\tau )$ and by induction in~\eqref{Algo1-lemma-3-eq-2} we deduce the result. Furthermore we have $|\gamma|<1$. Indeed, using the hypothesis $\eta_x< \beta\tau /(2\kappa_y)$ it is straightforward to verify that $\gamma<1$, while by construction $-1/2<\gamma$ since $\tau\in (0,1]$ and $\kappa_y\geq 1$. 
	\end{proof}

	\subsection{Proof of Theorem~\ref{thm_non_convex_strongly_concave_case_1}}
	The core of the proof consists in combining Lemma~\ref{Algo1-lemma-1} and Lemma~\ref{Algo1-lemma-3}.
	
	\begin{proof}[Proof of Theorem \ref{thm_non_convex_strongly_concave_case_1}]
		\gt{We first observe that the step-size condition of Theorem~\ref{thm_non_convex_strongly_concave_case_1} implies that of  Lemma~\ref{Algo1-lemma-1}. The step-size condition on $\eta_x$ given in \eqref{constraint_eta_x_algo1} is indeed $\eta_x < \min\left( 1/\lc,  \beta\tau / (2\sqrt{2} \kappa_y) \right)$ where $0 < \tau = \eta_y L_{yy} \leq 1$ and $\beta = \mu/(L_{xy}L_{yx})$. As a consequence we have $\eta_x\lc<1$, as assumed in Lemma~\ref{Algo1-lemma-1}. Moreover, we have $\eta_x< \beta \tau/ (2\sqrt{2} \kappa_y)<\beta\tau /(2\kappa_y)$ so the relation~\eqref{Algo1-lemma-3-main-eq} from Lemma~\ref{Algo1-lemma-3} holds with $|\gamma|<1$.}
		
		Plugging~\eqref{Algo1-lemma-3-main-eq} from Lemma~\ref{Algo1-lemma-3} into relation~\eqref{Algo1-lemma-1_main_eq}, we get
		
		\begin{align}\label{main_proof_algo_1_eq_1}
			\varphi(x_N) \leq & \ \varphi(x_0) - \frac{\eta_x}{2}\sum_{k=0}^{N-1}\norm{\nabla \varphi(x_k)}^2 + \frac{\eta_x}{2}L_{xy}^2\delta_0 \sum_{k=0}^{N-1}\gamma^k \notag \\
			&+\frac{\eta_x}{2} \left( \gamma - 1 +\frac{\tau}{2\kappa_y}\right)\sum_{k=0}^{N-1}\sum_{j=0}^{k-1}\gamma^{k-1-j}\norm{\nabla \varphi(x_j)}^2.
		\end{align}
		By Lemma~\ref{Algo1-lemma-3} we have $|\gamma|<1$, which implies $\sum_{k=0}^{N-1} \gamma^k < \sum_{k=0}^{+\infty} \gamma^k = \frac{1}{1-\gamma}>0$ and 
		\begin{equation*}
			\sum_{k=0}^{N-1} \sum_{j=0}^{k-1} \gamma^{k-1-j} \left\| \nabla \varphi(x_j) \right\|^2 
			= \sum_{j=0}^{N-2} \left( \sum_{m=0}^{N-2-j} \gamma^m \right) \left\| \nabla \varphi(x_j) \right\|^2
			\leq \frac{1}{1-\gamma} \sum_{k=0}^{N-1} \left\| \nabla \varphi(x_k) \right\|^2, 
		\end{equation*}
		where we exchange the order of summation and reindex via $m=k-1-j$. Then \eqref{main_proof_algo_1_eq_1} becomes
		
		\begin{equation}\label{main_proof_algo_1_eq_2}
			\varphi(x_N) \leq \varphi(x_0) - \frac{\eta_x}{2}\sum_{k=0}^{N-1}\norm{\nabla \varphi(x_k)}^2 + \frac{\eta_xL_{xy}^2\delta_0}{2(1-\gamma)} + \frac{\eta_x}{2} \left( \gamma - 1 +\frac{\tau}{2\kappa_y}\right)\frac{1}{1-\gamma}\sum_{k=0}^{N-1}\norm{\nabla \varphi(x_k)}^2.
		\end{equation}
		\gt{
			Noting that $\left( \gamma - 1 + \frac{\tau}{2\kappa_y} \right) \frac{1}{1-\gamma} = -1 + \frac{\tau}{2\kappa_y (1-\gamma)}$ we rewrite \eqref{main_proof_algo_1_eq_2} as follows
			\begin{equation*}\label{main_proof_algo_1_eq_3}
				\varphi(x_N) \leq \varphi(x_0) - \frac{\eta_x}{2} \underbrace{\left( 2 - \frac{\tau}{2\kappa_y(1-\gamma)} \right)}_{C_1}  \sum_{k=0}^{N-1}\norm{\nabla \varphi(x_k)}^2 + \underbrace{\frac{\eta_xL_{xy}^2\delta_0}{2(1-\gamma)}}_{C_2}.
			\end{equation*}
			To apply Proposition~\ref{sufficient_condition_epsilon_stationary_point} and conclude, it suffices to verify that   $C_1, C_2 > 0$. By Lemma~\ref{Algo1-lemma-3} we deduce that $1 - \gamma >0$; it follows that $C_2 >0$. On the other hand, recalling that, $ 1 - \gamma= \frac{\beta^2\tau^2 - 4\kappa_y^2 \eta_x^2}{2\kappa_y \beta^2\tau}$ we rewrite $C_1$ as follows: $C_1 = 2 - \frac{\beta^2\tau^2}{\beta^2\tau^2 - 4\kappa_y^2 \eta_x^2}$. By the step-size condition~\eqref{constraint_eta_x_algo1}, we have that $\eta_x^2 < \min\left( \frac{1}{\lc^2} , \frac{\tau^2 \beta^2}{8 \kappa_y^2} \right)$ which implies that $C_1 >0$. Therefore, the result follows from Proposition~\ref{sufficient_condition_epsilon_stationary_point}.}
\end{proof}

\gt{Note that neither condition on $\eta_x$ can be removed with our analysis. One-dimensional examples of the form $x^2 + xy + by^2$, $b \in \R$,  show that either term may be the smaller.}

\subsection{On the step-size}
Here we compare the step-size constraints of our Theorem~\ref{thm_non_convex_strongly_concave_case_1} with the ones proposed in the related literature under their associated  assumptions. 
\paragraph{Joint Lipschitz hypothesis}
We recall that if the coupling term $\Phi$ has a jointly Lipschitz continuous gradient,  then there exists $L_{\Phi}>0$ such that 
\begin{equation}\label{jointly_lipschitz_constant_def}
	\norm{\nabla \Phi(x,y) - \nabla \Phi(x',y')}
	\leq L_{\Phi} \norm{(x,y)-(x',y')}, \quad \forall x,x' \in \R^d \ \text{and} \  y,y' \in \R^n.
\end{equation}
This is the hypothesis made in \cite{lin2025two, bot2023lternatingproximalgradientstepsstochastic}. This property implies that $L_{xx} = L_{yy}=L_{xy}=L_{yx}$ so that our  constraint~\eqref{constraint_eta_x_algo1} \gt{on $\eta_x$ simplifies to $0<\eta_x < \eta_y \min \left( 1 /(\kappa_y + 1), 1/(2\sqrt{2}\kappa_y^2) \right)$}. Table~\ref{comparation_under_jointly_lipschitz} summarizes our results under the jointly Lipschitz hypothesis. \gt{We observe that since $\kappa_y \geq 1$, it follows that $16(\kappa_y +1)^2>3(\kappa_y +1)^2 > 2\sqrt{2}\kappa_y^2 > \kappa_y + 1$}. This shows that Theorem~\ref{thm_non_convex_strongly_concave_case_1} allows to  take larger step-sizes than \cite{lin2025two, bot2023lternatingproximalgradientstepsstochastic}.

\begin{table}[ht]
	\caption{Admissible step-sizes under the jointly Lipschitz setting for \algoone.}\label{comparation_under_jointly_lipschitz}
	\centering
	\begin{tabular}{c>{\centering\arraybackslash}p{3cm}>{\centering\arraybackslash}p{3cm}>{\centering\arraybackslash}p{6.1cm}}
		\toprule
		&  \cite{lin2025two} & \cite{bot2023lternatingproximalgradientstepsstochastic} & Ours  \\
		\midrule
		$\eta_y\leq $ & $1/\lp$& $1/\lp$& $ 1/\lp$ \\ 
		$\eta_x \leq$ & $\eta_y/(16(\kappa_y + 1)^2)$ &  $\eta_y/(3(\kappa_y + 1)^2)$ & \textcolor{black}{$\eta_y \min \left( 1/(\kappa_y + 1), 1/(2\sqrt{2}\kappa_y^2) \right)$} \\
		\midrule
	\end{tabular}
\end{table}

\paragraph{Block-wise Lipschitz hypothesis}
We share this assumption with \cite{Cohen2025AlternatingandParallelProximalGradientMethods} and present in Table~\ref{comparation_under_block_wise}  the respective step-sizes.  Two cases arise, and in both, our method imposes fewer constraints on step sizes. \gt{\emph{Case 1:} If $1/\lc$ attains the minimum in~\eqref{constraint_eta_x_algo1} on $\eta_x$, our bound reduces to $1/\lc$ and it is larger than the one in \cite{Cohen2025AlternatingandParallelProximalGradientMethods}, since $\mu L_{xy}^{2} + 2\kappa_y(2\kappa_y+1)L_{yx}^{2} > 0$}. \emph{Case 2:} We here assume that  $L_{xy} = L_{yx}$ in order to carry out the comparison\footnote{If $\Phi$ is of class $C^{2}$, then by Clairaut's theorem this assumption is straightforwardly satisfied.}. Let assume that the minimum is reached at $\eta_y L_{yy} \mu /(2\sqrt{2}\kappa_y L_{xy}L_{yx})$.  In this second setting, recalling that $\lc = L_{xx} + L_{xy}L_{yx}/\mu$ and observing that $\mu L_{xy}^{2} + \mu L_{xx} + L_{xy}^{2} + \bigl(4\kappa_y^{2} + 2\kappa_y(1-\sqrt{2})\bigr)L_{xy}^{2} > 0$, we deduce that our algorithm still allows for larger step-sizes than in \cite{Cohen2025AlternatingandParallelProximalGradientMethods}. Nevertheless, if $L_{xy} \neq L_{yx}$, we cannot generally determine which bound on $\eta_x$ is tighter without additional assumptions.

\begin{table}[ht]
	\caption{Admissible step-sizes under the block-wise Lipschitz setting for \algoone.}\label{comparation_under_block_wise}
	\centering
	\begin{tabular}{c>{\centering\arraybackslash}p{6cm}>{\centering\arraybackslash}p{7.3cm}}
		\toprule
		& \cite{Cohen2025AlternatingandParallelProximalGradientMethods} & Ours  \\
		\midrule
		$\eta_y\leq $ & $ 1/L_{yy}$& $ 1/L_{yy}$\\
		$\eta_x \leq$&  $\mu / ( \ \mu(L_{xy}^2 + \lc) + 2\kappa_y(2\kappa_y+1)L_{yx}^2 \ )$ & $ \gt{\min \left( 1/\lc, \eta_y L_{yy}\mu / ( \ 2\sqrt{2}\kappa_y L_{xy}L_{yx} \ ) \right)}$ \\
		\midrule
	\end{tabular}
\end{table}



\section{Convergence analysis for \algotwo}\label{section_convergence_analysis_for_algorithm_2}

We recall that \algotwo is defined $\forall k\geq 0$ by
\begin{align}
	x_{k+1} &= \prox{\eta_x \Phi(\cdot,y_{k})}(x_k)  \label{algo_2_update_x_sect_4} \\
	y_{k+1} &= \prox{\eta_y h}\left(y_k + \eta_y \nabla_y \Phi(x_{k+1}, y_k)\right). \label{algorithm_prox_da_a_2_sect_4}
\end{align}
Note that the $y$-update in \algotwo is the same as in \algoone. The difference between both algorithms lies in the $x$-update: an explicit gradient step is performed in~\eqref{algo_1_update_x_sect_3}, whereas \algotwo realizes a proximal step~\eqref{algo_2_update_x_sect_4} on the coupling function $\Phi$ that is $\rho$-weakly convex.  Hence, to have a well-posed proximal step (see Definition~\ref{proximal_operator}), we make the following assumption.
\begin{assumption}\label{assumption_4}
	The step-size $\eta_x$ satisfies $\eta_x \rho <1$. Hence the proximal step of~\eqref{algo_2_update_x_sect_4} is single-valued.
\end{assumption}

Let us now state the main result of this section.
\begin{theorem}\label{main_thm_algo_2}
	Let Assumptions~\ref{assumption_1}, \ref{assumption_2}, \ref{assumption_3}, and \ref{assumption_4} hold true. For any initial point $(x_0,y_0) \in \R^d \times \R^n$ we consider the sequence $(x_k,y_k)_{k \geq 1}$ generated by \algotwo with step-sizes 
	\gt{
		\begin{equation}\label{algo2_step-sizes_condition}
			0 < \eta_y \leq \frac{1}{L_{yy}} \quad \text{and} \quad 0< \eta_x < \min\left( \frac{\eta_y L_{yy} \mu}{ 2 L_{xy}L_{yx}(\sqrt{2}\kappa_y + \eta_yL_{yy})} , \frac{1}{\rho} \right).
	\end{equation}}
	Then, for every $\varepsilon>0$ there exists $k = \mathcal{O} (\varepsilon^{-2})$ such that $\norm{\nabla \varphi(x_k)}<\varepsilon$.
\end{theorem}

Recall that we parameterize the step-size $\eta_y$ as $\eta_y = \tau / L_{yy}$ for some $\tau \in (0,1]$. 
For simplicity, we write $(x_k, y_k)_{k \geq 1}$ to avoid explicitly including $\tau$ in the notation for the sequences generated by \algotwo.

In order to prove Theorem~\ref{main_thm_algo_2}, we first provide key results which form the basis of the convergence analysis. 
\begin{notation}
	We recall that the mapping $\varphi$ is defined as $\varphi(x) = \max_{y \in \R^n} \ \left( \Phi(x,y)-h(y) \right)$ and that under strong concavity (Assumption~\ref{assumption_1}), $y^*(x) = \operatorname{argmax}_{y \in \R^n}\left( \Phi(x,y)-h(y)\right)$.    We also recall that we denote
	$y^*_k=y^*(x_k)$, $\delta_k=\norm{y^*(x_k)-y_k}^2$ and $\beta = \mu /(L_{xy}L_{yx})$.
\end{notation}
\begin{lemma}\label{Algo2_lemma_1} 
	Let Assumptions \ref{assumption_1} and \ref{assumption_2} hold true and $(x_k,y_k)_{k \geq 1}$ be a sequence generated by \algotwo and $N \geq 1$. \gt{If $\eta_x \rho < 1$, then we have
	\begin{equation}\label{Algo2_lemma_1_eq_statement}
			\varphi(x_{N}) \leq \varphi(x_0) - \frac{\eta_x}{2} \sum_{k=0}^{N-1}\norm{\nabla \varphi(x_{k+1})}^2 + \frac{\eta_x}{2} L_{xy}^2 \sum_{k=0}^{N-1} \norm{y^*(x_{k+1}) -y_k}^2.
	\end{equation}}
\end{lemma}
The proof of this result is presented in Appendix~\ref{app:lem_alg2}. Note that the right-hand side of \eqref{Algo2_lemma_1_eq_statement} contains $\norm{y^*(x_{k+1}) - y_k}^2$ whereas for \algoone the corresponding term in relation~\eqref{Algo1-lemma-1_main_eq} depends on $\delta_k = \norm{y^*(x_{k}) - y_k}^2$. This subtle distinction complicates the analysis:
while $\delta_k$ explicitly measures the gap between the inner problem's maximizer $y^*(x_k)$ at $x_k$ and the current iterate $y_k$,
the analogous interpretation does not hold for the term in~\eqref{Algo2_lemma_1_eq_statement}.
This discrepancy arises from the fact that the proximal step can be seen as an ``implicit'' step, which introduces an index shift and alters the relationship between the terms.

\subsection{Control of \texorpdfstring{$\norm{y^*_{k+1}-y_k}^2$}{y*{k+1}-yk}}

The following result shows that the term $\norm{y^*_{k+1}-y_k}^2$ can be controlled by considering an additional constraint on the step-size $\eta_x$.

\begin{lemma}\label{Algo2-lemma-3}
	Let Assumptions~\ref{assumption_1}, \ref{assumption_2}, and \ref{assumption_4} hold true and $(x_k,y_k)_{k \geq 1}$ be a sequence generated by \algotwo. Then, for all $\theta>0$ and $\eta_x^2 < \beta^2 / \left( 2(1+1/\theta) \right)$ we have that
	\begin{equation*}
		\norm{y_{k+1}^* - y_k}^2 \leq \frac{(1+\theta)\beta^2}{\beta^2- 2 \eta^2_x  \left(1 + 1/\theta\right)}\delta_k + \frac{2 \eta^2_x  \left(1 + 1/\theta \right)}{L_{xy}^2(\beta^2 - 2\eta_x^2(1+1/\theta))}\norm{\nabla \varphi(x_{k+1})}^2.
	\end{equation*}
	\begin{proof}
		Applying triangular and  Young inequalities to $\norm{y^*_{k+1} -y_k}^2$ we get for any $\theta>0$ 
		\begin{align}
			\norm{y^*_{k+1} -y_k}^2 
			&\leq \norm{y^*_{k+1}-y^*_{k}}^2 + 2\norm{y^*_{k+1}-y^*_k} \norm{y^*_k-y_k} + \norm{y^*_k-y_k}^2 \notag \\
			&\leq \left( 1 +\frac{1}{\theta}\right)\norm{y^*_{k+1}-y^*_{k}}^2 + (1+\theta)\norm{y^*_k-y_k}^2. \label{Algo2-lemma-3_eq_0}
		\end{align}
		Using the  $L_{yx}/\mu$-Lipschitz continuity of  $y^*$ w.r.t. $x$ (Lemma~\ref{Lipschitz_continuity_of_the_solution_mapping}) and leveraging the first-order optimality condition of the proximal step~\eqref{algo_2_update_x_sect_4}, $x_{k+1}-x_k=-\eta_x \nabla_x \Phi(x_{k+1},y_k)$, relation~\eqref{Algo2-lemma-3_eq_0} becomes
		\begin{equation}\label{Algo2-lemma-3_eq_1}
			\norm{y^*_{k+1} -y_k}^2 \leq \left( 1 +\frac{1}{\theta}\right) \frac{L_{yx}^2\eta_x^2}{\mu^2} \norm{\nabla_x \Phi(x_{k+1},y_k)}^2 + (1+\theta)\delta_k. 
		\end{equation}
		Applying the \gt{Cauchy-Schwarz} inequality, Proposition~\ref{Smoothness_of_the_max_function} and Assumption~\ref{assumption_2}(i)(b), 
		we get 
		\begin{align*}
			\norm{\nabla_x \Phi(x_{k+1},y_k)}^2 &\leq 2 \left(\norm{\nabla \varphi(x_{k+1}) - \nabla_x \Phi(x_{k+1},y_k)}^2 + \norm{\nabla \varphi(x_{k+1})}^2\right) \\
			&\leq 2\left( L_{xy}^2\norm{y^*_{k+1}-y_k}^2 +  \norm{\nabla \varphi(x_{k+1})}^2\right).
		\end{align*}
		Plugging this last inequality into \eqref{Algo2-lemma-3_eq_1} and reorganizing the terms we get
		\begin{equation*}
			\left(1- 
			\frac{2L_{yx}^2L_{xy}^2\eta_x^2}{\mu^2} \left(1 + \frac{1}{\theta}\right) \right)\norm{y^*_{k+1} -y_k}^2  \leq  \frac{2L_{yx}^2\eta_x^2}{\mu^2} \left(1 + \frac{1}{\theta}\right)\norm{\nabla \varphi(x_{k+1})}^2 + (1+\theta)\delta_k.
		\end{equation*}
		Defining $\beta=\mu/(L_{xy}L_{yx})$ yields
		\begin{equation*}
			\norm{y^*_{k+1} -y_k}^2 \left(1- 
			\frac{2\eta_x^2}{\beta^2} \left(1 + \frac{1}{\theta}\right) \right) \leq  \frac{2\eta_x^2}{\beta^2L_{xy}^2}  \left(1 + \frac{1}{\theta}\right)\norm{\nabla \varphi(x_{k+1})}^2 + (1+\theta)\delta_k.
		\end{equation*}
		We finally take $\eta_x^2 < \beta^2 / \left( 2(1+1/\theta) \right)$, which implies $1 - \frac{2\eta_x^2}{\beta^2} \left(1 + \frac{1}{\theta}\right) >0$ in the above inequality and allows to recover the statement of the lemma. 
	\end{proof}
\end{lemma}

Recall that, since the $y$ variable is updated in the same way for both \algoone and \algotwo, we can apply Lemma~\ref{Algo1-lemma-2} to \algotwo and thereby obtain that:
\begin{equation}\label{general_bound_delta_k}
	\delta_{k+1} \leq \left( 1 - \frac{\tau}{2\kappa_y} \right) \delta_k + \frac{\kappa_y}{\tau} \norm{ y_{k+1}^* - y_k^* }^2 \quad \forall \ 0 < \tau \leq 1.
\end{equation}
We now refine~\eqref{general_bound_delta_k} specifically for \algotwo.

\begin{lemma}[Bound of $\delta_{k}$ for \algotwo]\label{Algo2-lemma-2}
	Let Assumptions~\ref{assumption_1}, \ref{assumption_2}, \ref{assumption_3}, and \ref{assumption_4} hold true and $(x_k,y_k)_{k \geq 1}$ be the sequence generated by \algotwo. Then, for every $k \in \N_0$ we have
	\begin{equation*}
		\delta_k \leq \gamma^k \delta_0 + \left( \gamma - 1 + \frac{\tau}{2\kappa_y}\right)\frac{1}{(1+\theta)L_{xy}^2} \sum_{j=0}^{k-1}\gamma^{k-1-j}\norm{\nabla \varphi(x_{j+1})}^2 \quad \forall \ 0 < \tau \leq 1,
	\end{equation*}
	where $\gamma=1 -\frac{\tau}{2\kappa_y} + \frac{2\kappa_y\eta_x^2(1+\theta)}{\tau(\beta^2 - 2\eta_x^2(1+1/\theta))}$. Furthermore, if $\eta_x< \beta\tau /(\sqrt{4\kappa_y^2(1+\theta) +2 \tau^2(1+1/\theta)})$, then $|\gamma|<1$.
\end{lemma}

\begin{proof}
	Using the Lipschitz property of $y^*$ (Lemma \ref{Lipschitz_continuity_of_the_solution_mapping}) and the first order optimality condition of the \gt{proximal} operator~\eqref{algo_2_update_x_sect_4} (\textit{i.e.,} $x_{k+1} - x_k = -\eta_x \nabla_x \Phi(x_{k+1},y_k)$) in~\eqref{general_bound_delta_k} we get
	\begin{equation}\label{eq_1_delta_k_for_algo_2}
		\delta_{k+1} \leq \left( 1-\frac{\tau}{2\kappa_y}\right) \delta_k + \frac{\kappa_yL_{yx}^2\eta_x^2}{\tau\mu^2}  \norm{\nabla_x \Phi(x_{k+1},y_k)}^2.
	\end{equation}
	Now, by Young's inequality  we have that 
	\begin{equation}\label{eq_2_delta_k_for_algo_2}
		\norm{\nabla_x \Phi(x_{k+1},y_k)}^2 \leq 2 \norm{\nabla_x \Phi(x_{k+1},y_k) - \nabla \varphi(x_{k+1})}^2 + 2\norm{\nabla \varphi(x_{k+1})}^2.
	\end{equation}
	Plugging~\eqref{eq_2_delta_k_for_algo_2} into~\eqref{eq_1_delta_k_for_algo_2} we get
	\begin{align*}
		\delta_{k+1} \leq &\left( 1-\frac{\tau}{2\kappa_y}\right) \delta_k + \frac{2\kappa_yL_{yx}^2\eta_x^2}{\tau\mu^2}  \norm{\nabla \varphi(x_{k+1})}^2\\ 
		&+ \frac{2\kappa_yL_{yx}^2\eta_x^2}{\tau\mu^2} \norm{\nabla_x \Phi(x_{k+1},y_k) -\nabla \varphi(x_{k+1})}^2.
	\end{align*}
	By Proposition \ref{Smoothness_of_the_max_function} and Assumption~\ref{assumption_2}(i) we deduce that $\norm{\nabla \varphi(x_{k+1})-\nabla_x \Phi(x_{k+1},y_k)}^2 \leq L_{xy}^2 \norm{y^*_{k+1}-y_k}^2$. Then, we obtain
	\begin{align}
		\delta_{k+1} \leq &\left( 1- \frac{\tau}{2\kappa_y}\right) \delta_k + \frac{2\kappa_yL_{yx}^2\eta_x^2}{\tau\mu^2}  \norm{\nabla \varphi(x_{k+1})}^2 + \frac{2\kappa_yL_{yx}^2\eta_x^2L_{xy}^2}{\tau\mu^2}  \norm{y^*_{k+1}-y_k}^2 .\notag 
	\end{align}
	Using $\beta=\mu/(L_{xy}L_{yx})$, the previous inequality becomes
	\begin{align}\label{eq_3_delta_k_for_algo_2}
		\delta_{k+1} \leq \left( 1- \frac{\tau}{2\kappa_y}\right) \delta_k + \frac{2\kappa_y \eta_x^2}{ \beta^2\tau L_{xy}^2} \norm{\nabla \varphi(x_{k+1})}^2 + \frac{2\kappa_y\eta_x^2}{\beta^2\tau } \norm{y^*_{k+1}-y_{k}}^2.
	\end{align}
	Plugging Lemma \ref{Algo2-lemma-3} into \eqref{eq_3_delta_k_for_algo_2} we deduce
	\begin{equation}\label{eq_3_delta_k_for_algo_2_new}
		\delta_{k+1} \leq \left(1 -\frac{\tau}{2\kappa_y} + \frac{2\kappa_y\eta_x^2(1+\theta)}{\tau(\beta^2 - 2\eta_x^2(1+1/\theta))} \right)\delta_k + \frac{2\kappa_y\eta_x^2}{\tau L_{xy}^2(\beta^2-2\eta_x^2(1+1/\theta))}\norm{\nabla \varphi(x_{k+1})}^2.
	\end{equation}
	Denoting $\gamma=1 -\frac{\tau}{2\kappa_y} + \frac{2\kappa_y\eta_x^2(1+\theta)}{\tau(\beta^2 - 2\eta_x^2(1+1/\theta))}$ and by induction in~\eqref{eq_3_delta_k_for_algo_2_new} we deduce that
	\begin{equation*}
		\delta_k \leq \gamma^k \delta_0 + \left( \gamma - 1 + \frac{\tau}{2\kappa_y}\right)\frac{1}{(1+\theta)L_{xy}^2} \sum_{j=0}^{k-1}\gamma^{k-1-j}\norm{\nabla \varphi(x_{j+1})}^2.
	\end{equation*}
	We finally observe that $\gamma<1$ for $\eta_x< \beta\tau /(\sqrt{4\kappa_y^2(1+\theta) +2 \tau^2(1+1/\theta)})$, while by construction $-1/2<\gamma$ since $0<\tau\leq1\leq\kappa_y$.
\end{proof}

Note that the constraint on $\eta_x$ in Lemma~\ref{Algo2-lemma-2} implies the one of Lemma~\ref{Algo2-lemma-3} since
\begin{equation}\label{algo_2_eta_x_relation_before_theta_choice}
	\eta^2_x < \frac{\tau^2 \beta^2}{2 ( 2\kappa_y^2(1+\theta) + \tau^2(1+1/\theta) )}=\frac{ \beta^2}{4\kappa_y^2(1+\theta)/\tau^2 + 2(1+1/\theta) } < \frac{\beta^2}{2(1+1/\theta)}.
\end{equation}
To allow for the  largest admissible values of $\eta_x$, we fix $\theta= \tau / (\sqrt{2}\kappa_y)$ in order to maximize the mapping $\theta \mapsto \theta/(2\kappa_y^2\theta^2 + (2\kappa_y^2+\tau^2)\theta + \tau^2)$ in~\eqref{algo_2_eta_x_relation_before_theta_choice}. With this value, we obtain:
$\eta_x<\beta\tau /(\sqrt{2}(\sqrt{2}\kappa_y+\tau))$.

\begin{corollary}\label{cor_sect_4}
	Under assumptions \ref{assumption_1}-\ref{assumption_4}  we have that for $\eta_x < \beta\tau /(\sqrt{2}(\sqrt{2}\kappa_y + \tau))$
	\begin{align}
		\norm{y^*_{k+1}-y_k}^2 \leq& \frac{(\sqrt{2}\kappa_y + \tau)\beta^2\tau\delta_0}{\sqrt{2}\kappa_y ( \beta^2\tau - 2(\sqrt{2}\kappa_y + \tau)\eta^2_x ) }\gamma^k  +  \frac{2(\sqrt{2}\kappa_y + \tau)\eta_x^2}{L_{xy}^2(\beta^2\tau - 2\eta_x^2(\sqrt{2}\kappa_y+\tau))}\norm{\nabla \varphi(x_{k+1})}^2 \notag \\
		&+ \left( \gamma -1 + \frac{\tau}{2\kappa_y}\right)\frac{\beta^2\tau}{L_{xy}^2(\beta^2\tau -2\eta_x^2(\sqrt{2}\kappa_y + \tau))}\sum_{j=0}^{k-1}\gamma^{k-1-j}\norm{\nabla \varphi(x_{j+1})}^2. \label{remark_sect_4_eq_1}
	\end{align}
	where $\displaystyle \gamma= 1 - \frac{\tau}{2\kappa_y} + \frac{\sqrt{2}(\sqrt{2}\kappa_y + \tau)\eta_x^2}{\beta^2\tau - 2(\sqrt{2}\kappa_y + \tau)\eta_x^2}$ with $|\gamma|<1$. 
\end{corollary}

The proof of this result simply consists in combining Lemma~\ref{Algo2-lemma-2} with Lemma~\ref{Algo2-lemma-3} for $\theta=\tau/(\sqrt{2}\kappa_y)$.

\subsection{Proof of Theorem~\ref{main_thm_algo_2}}
The proof strategy consists in combining Corollary~\ref{cor_sect_4} and Lemma~\ref{Algo2_lemma_1}.

\begin{proof}[Proof of Theorem \ref{main_thm_algo_2}]
	
	First of all, we note that the step-size condition~\eqref{algo2_step-sizes_condition} of Theorem~\ref{main_thm_algo_2} implies that of Corollary~\ref{cor_sect_4}. Recall that the step-size condition on $\eta_x$ given in~\eqref{algo2_step-sizes_condition} is indeed $\eta_x < \min\left(  \beta\tau / ( 2(\sqrt{2} \kappa_y + \tau)~) , 1/\rho \right)$ where $0 < \tau = \eta_y L_{yy} \leq 1$ and $\beta = \mu/(L_{xy}L_{yx})$. As a consequence we have $\eta_x \rho < 1$, as assumed in Lemma~\ref{Algo2_lemma_1}. Moreover, we have $\eta_x < \beta \tau / (2(\sqrt{2} \kappa_y + \tau)~) < \beta \tau /(\sqrt{2}(\sqrt{2} \kappa_y + \tau)~)$ so the relation~\eqref{remark_sect_4_eq_1} from Corollary~\ref{cor_sect_4} holds with $\lvert \gamma \rvert < 1$.
	
	Now, plugging~\eqref{remark_sect_4_eq_1} into~\eqref{Algo2_lemma_1_eq_statement} we have that
	\begin{align}\label{main_proof_algo_2_eq_0}
		\varphi(x_N) &\leq \varphi(x_0) - \frac{\eta_x}{2}\sum_{k=0}^{N-1}\norm{\nabla \varphi(x_{k+1})}^2 \notag  
		+ \frac{\eta_x(\sqrt{2}\kappa_y + \tau)\beta^2\tau L_{xy}^2\delta_0}{2\sqrt{2}\kappa_y(\beta^2\tau - 2\eta_x^2(\sqrt{2}\kappa_y +\tau))}\sum_{k=0}^{N-1}\gamma^k \notag \\
		&+ \frac{\eta_x}{2}\left( \gamma -1 +\frac{\tau}{2\kappa_y}\right)\frac{\beta^2\tau}{\gt{\beta^2\tau - 2\eta_x^2(\sqrt{2}\kappa_y +\tau)}}\sum_{k=0}^{N-1}\sum_{j=0}^{k-1}\gamma^{k-1-j}\norm{\nabla \varphi(x_{j+1})}^2 \notag \\
		&+\frac{\eta_x}2 \frac{2(\sqrt{2}\kappa_y+\tau)\eta_x^2}{\beta^2\tau -2\eta_x^2(\sqrt{2}\kappa_y + \tau)}\sum_{k=0}^{N-1}\norm{\nabla \varphi(x_{k+1})}^2.
	\end{align}
	By Corollary~\ref{cor_sect_4} we have  $|\gamma|<1$, which implies that $\sum_{k=0}^{N-1} \gamma^k < \sum_{k=0}^{+\infty} \gamma^k = \frac{1}{1-\gamma}$ and 
	\begin{align*}
		\sum_{k=0}^{N-1} \sum_{j=0}^{k-1} \gamma^{k-1-j} \left\| \nabla \varphi(x_{j+1}) \right\|^2 
		& = \sum_{j=0}^{N-2} \left( \sum_{m=0}^{N-2-j} \gamma^m \right) \left\| \nabla \varphi(x_{j+1}) \right\|^2 \\
		&\leq \frac{1}{1-\gamma} \sum_{k=0}^{N-1} \left\| \nabla \varphi(x_{k+1}) \right\|^2,
	\end{align*}
	where we exchange the order of summation and reindex via $m=k-1-j$. Then \eqref{main_proof_algo_2_eq_0} becomes
	\begin{align}\label{main_proof_algo_2_eq_1}
		&\varphi(x_N)\leq \varphi(x_0) -\frac{\eta_x}{2}\sum_{k=0}^{N-1}\norm{\nabla \varphi(x_{k+1})}^2 + \frac{\eta_x(\sqrt{2}\kappa_y+\tau)\beta^2\tau  L_{xy}^2\delta_0}{2\sqrt{2}\kappa_y(\beta^2\tau -2\eta_x^2(\sqrt{2}\kappa_y +\tau))(1-\gamma)} \notag\\ &+\frac{\eta_x}{2}\frac1{(\beta^2\tau-2\eta_x^2(\sqrt{2}\kappa_y+\tau))}\left(\left(\gamma-1+\frac{\tau}{2\kappa_y}\right)\frac{\beta^2\tau}{1-\gamma}+2\eta_x^2(\sqrt{2}\kappa_y+\tau)\right)\sum_{k=0}^{N-1}\norm{\nabla \varphi(x_{k+1})}^2. 
	\end{align}
	\gt{Since $\left(\gamma -1 +\frac{\tau}{2\kappa_y} \right)\frac{1}{1-\gamma}= -1 + \frac{\tau}{2\kappa_y(1-\gamma)} $, it follows that
		\begin{align}\label{main_proof_algo_2_eq_2}
			&\frac{1}{\beta^2\tau-2\eta_x^2(\sqrt{2}\kappa_y+\tau)}\left(\left(\gamma-1+\frac{\tau}{2\kappa_y}\right)\frac{\beta^2\tau}{1-\gamma}+2\eta_x^2(\sqrt{2}\kappa_y+\tau)\right) = \notag \\  
			=&- 1 + \frac{\tau^2\beta^2}{2\kappa_y(1-\gamma)(\beta^2\tau -2\eta_x^2(\sqrt{2}\kappa_y + \tau))}.
		\end{align}
		Combining relations~\eqref{main_proof_algo_2_eq_1} and~\eqref{main_proof_algo_2_eq_2} we get
		\begin{align}\label{main_proof_algo_2_eq_4}
			\varphi(x_N) \leq &\varphi(x_0) - \frac{\eta_x}{2} \underbrace{\left( 2 - \frac{\tau^2\beta^2}{2\kappa_y(1-\gamma)(\beta^2\tau-2\eta_x^2(\sqrt{2}\kappa_y + \tau))}\right)}_{C_1}\sum_{k=0}^{N-1}\norm{\nabla \varphi(x_{k+1})}^2 \notag \\
			&+ \underbrace{\frac{\eta_x(\sqrt{2}\kappa_y+\tau)\beta^2\tau L_{xy}^2\delta_0}{2\sqrt{2}\kappa_y(\beta^2\tau-2\eta_x^2(\sqrt{2}\kappa_y + \tau))(1-\gamma)}}_{C_2}.
		\end{align}
		To apply Proposition~\ref{sufficient_condition_epsilon_stationary_point} and conclude, it suffices to verify that   $C_1, C_2 > 0$. From Corollary~\ref{cor_sect_4}, we have that $|\gamma|<1$, and its condition on $\eta_x$ implies the one required in Lemma~\ref{Algo2-lemma-3}, as shown in~\eqref{algo_2_eta_x_relation_before_theta_choice}. Therefore, $\beta^2\tau-2\eta_x^2(\sqrt{2}\kappa_y + \tau)>0$ which yields $C_2>0$. On the other hand, noting that $$1-\gamma = \frac{\beta^2 \tau^2 - 2\eta_x^2(\sqrt{2}\kappa_y + \tau)^2}{2\kappa_y(\beta^2 \tau - 2\eta_x^2(\sqrt{2}\kappa_y + \tau))},$$ 
		we rewrite $C_1$ as follows
		\begin{equation*}
			C_1 = 2 - \frac{\beta^2 \tau^2}{\beta^2 \tau^2 - 2\eta_x^2(\sqrt{2}\kappa_y + \tau)^2}.
		\end{equation*}
		By the step-size condition~\eqref{algo2_step-sizes_condition}, we have that $\eta_x^2 < \min\left( \frac{\tau^2 \beta^2}{ 4(\sqrt{2}\kappa_y + \tau)^2} , \frac{1}{\rho^2} \right)$ which implies $C_1 >0$. Therefore, the result follows from Proposition~\ref{sufficient_condition_epsilon_stationary_point}.}
\end{proof}

Note that, within our approach, the set of conditions on the step-size $\eta_x$ cannot be further reduced. As in the analysis of \algoone, it is indeed possible to construct one-dimensional examples of the form $-ax^2+bxy-cy^2$, with $a,b,c>0$, for which either the first or the second condition is the more restrictive one.



{
	\color{black}
	
	\section{Experiments}\label{section_numerical_experiments}
	
	In this section, we consider \gt{four} nonconvex–strongly concave examples on which we evaluate the performance of \algoone and \algotwo. We now present the setting.
	
	\subsection{Setting\label{sec:setting}}
	
	\paragraph{Toy experiment}
	Inspired by~\cite{bot2023lternatingproximalgradientstepsstochastic}, we consider the problem
	\begin{equation}\label{toy_example}
		\min_{x \in \R} \max_{y \in \R} \left\{ g(x) +xy - 0.5 y^2  \right\}, \quad \text{where} \quad 
		g(x) = \begin{cases}
			0.5 - x^2 & \text{if} \ \lvert x \rvert \leq  0.5, \\
			\left( \lvert x\rvert -1 \right)^2  & \text{otherwise}.
		\end{cases}
	\end{equation}
	We define $\Phi(x,y)=g(x) + xy - 0.5y^2$ so that $\Phi$ is $C^1$, $\Phi(\cdot,y)$ is $\rho=2$-weakly convex w.r.t. $x$ and $\Phi(x,\cdot)$ is $\mu=1$-strongly concave w.r.t. $y$. For \algotwo, $\prox{\Phi(\cdot,y_k)}$ has a closed-form expression given in Appendix~\ref{sec::additionalExp}. The mapping $\varphi$ has two stationary points at $\pm 2/3$.
	
	\paragraph{Dirac-GAN Problem}
	We consider the nonconvex-concave min-max formulation \cite{Xu2023projalgo} of the Dirac-GAN problem introduced by \cite{mescheder2018training}: 
	\begin{equation}\label{strongly_concave_dirac_gan}
		\min_{x \in \R} \max_{y \in \R} \ \Phi(x,y) = -\log(1 + \exp(-xy)) + \log(2)   - \frac{1}{4} \ y^2,
	\end{equation}
	where we added $- \frac{1}{4} \ y^2$ (so $\mu=1/2$) to make the problem strongly concave in $y$. \emph{We consider two versions of the problem}, one where we see this additional term as part of $\Phi$ (so $h\equiv 0$) and one where $h(y)=\frac{1}{4} \ y^2$ is the regularizer associated with the coupling as $\Tilde{\Phi}(x,y) = -\log(1 + \exp(-xy)) + \log(2)$. The latter only affects the block-wise Lipschitz constant and condition number  w.r.t. $y$, which we denote by $\Tilde{L}_{yy}$ and $\Tilde{\kappa}_y := \Tilde{L}_{yy}/\mu$. Specifically, the value reported in Table~\ref{tab:lipschitz} changes from $L_{yy} = 4.5$ to $\Tilde{L}_{yy} = 4$, and consequently from $\kappa_y = 9$ to $\Tilde{\kappa}_y = 8$.

	\paragraph{Robust Logistic Regression}
	We also apply the algorithms to nonconvex robust logistic regression (RLR) on three \textbf{LIBSVM} datasets\footnote{https://www.csie.ntu.edu.tw/~cjlin/libsvmtools/datasets/} as studied in~\cite{lin2025two}. 
	Let $A \in \R^{n \times d}$ be a data matrix whose $i^{th}$ row $a_i$ contains the features of the $i^{th}$ sample, and let $b_i \in \{-1, 1 \}$ be the corresponding label. The problem reads:
	\begin{equation}\label{rlr_model}
		\min_{x \in \R^d} \max_{y \in \R^n} \ \frac{1}{n} \left( \sum_{i=1}^n y_i \log(1 + \exp(-b_i a_i^T x)) \right) - \frac{\lambda_1}{2} \norm{n y - \textbf{1}_{n}}^2 + \lambda_2 \left( \sum_{i=1}^d \frac{\alpha x_i^2}{1 + \alpha x_i^2}\right) - \ifct_{\triangle},
	\end{equation}
	where $\triangle= \{y \in \R^n_{+} : y^T \textbf{1}_n = 1 \} $, $\lambda_1 = 1/n^2$, $\lambda_2 = 10^{-2}$, and $\alpha = 10$ as in  \cite{kohler2017sub}.
	We define $h(y) = \ifct_{\triangle}$ so the remaining terms correspond to the coupling $\Phi$, which is $\mu = \lambda_1 n^2$-strongly concave w.r.t. $y$ and $\rho = \alpha \lambda_2 / 2$- weakly convex w.r.t. $x$. 
	
    }
	\paragraph{Plug-and-Play Image Restoration} The final problem is a Plug-and-Play (PnP) image experiment, which is the main motivation for \algotwo.
	Imaging inverse problems are often formulated as $\min_{x \in \R^d} \lambda f(x) + g(x)$ where $f$ is a data-fidelity term, $g$ is a weakly-convex regularizer and $\lambda > 0$ a parameter that controls the regularization.  
	Here we consider $f(x) = 1/(2\sigma^2)\norm{Ax-b}^2$ which corresponds to the fidelity to a degraded observation $b=Ax^* + \gt{\xi}$ through a linear operator $A: \R^d \to \R^n$, an additive  Gaussian noise $\gt{\xi}$ with standard deviation $\sigma$, and a clean image $x^*$.
	By dualization, we can actually reformulate the problem in the min-max form:
	\begin{equation}\label{min_max_formulation_pnp}
		\min_{x \in \R^d} \max_{y \in \R^n} \quad \langle Ax-b, y \rangle - \frac{\sigma^2}{2\lambda}\norm{y}^2 + g(x).
	\end{equation}
	Considering the coupling $\Phi(x,y)=\langle Ax-b, y \rangle - \frac{\sigma^2}{2\lambda}\norm{y}^2 + g(x)$,  \algotwo writes
	\setcounter{algorithm}{1}
	\begin{equation*}
		\begin{cases}
			x_{k+1} &= \prox{\eta_x g}(x_k - \eta_x A^Ty_k)
			\\
			y_{k+1} &= y_k + \eta_y\left(Ax_{k+1} - b - \frac{\sigma^2}{\lambda}y_k \right), 
		\end{cases}
	\end{equation*}
	which can be seen as a nonconvex primal-dual method~\cite{chambolle2011first,shumaylov2024weakly}. In our experiments, we consider the Plug-and-Play framework of~\cite{Hurault2022ProximalDF}, in which the regularization $g$ is replaced by a function $\phi_{\sigma}$ parameterized by a neural network. {\color{black} We perform super-resolution and deblurring experiments, for which additional details are provided in Section~\ref{sec::additionalExp} of the Appendix.

	\paragraph{Problem-dependent parameters} 
	The values of the Lipschitz constants, strong-convexity constants and weak-convexity modulus are reported in Table~\ref{tab:lipschitz}. The Dirac-GAN problem does not satisfy Assumptions~\ref{assumption_2} and~\ref{assumption_3} globally, but only locally. Therefore we estimated the parameters $L_{xx}$ and $L_{yy}$ on the compact set $[-4,4]^{d\times n}$ containing the iterates. The stopping criterion $\varepsilon$ is set to $10^{-4}$ for Robust Logistic Regression~\eqref{rlr_model}, $10^{-5}$ for the Dirac-GAN problem~\eqref{strongly_concave_dirac_gan}, and $10^{-10}$ for the Toy example~\eqref{toy_example}.
	\begin{table}[ht]
		\centering
		\caption{Parameters of the four
			problems considered.}
		\label{tab:lipschitz}
		\small
		\setlength{\tabcolsep}{4.5pt}
		\begin{tabular}{@{}lllll@{}}
			\toprule
			& \textbf{Toy} \eqref{toy_example}
			& \textbf{D-GAN} \eqref{strongly_concave_dirac_gan}
			& \textbf{Rob. log. reg.} \eqref{rlr_model}
			& \textbf{PnP} \eqref{min_max_formulation_pnp} \\
			\midrule
			$L_{xx}$
			& $2$
			& $y^2/4 \rightarrow 4$
			& $\tfrac{1}{4n}\max_i\norm{a_i}^2+\tfrac15$
			& $\tfrac{L}{1-L}$ \\[2pt]
			$L_{xy}=L_{yx}$
			& $1$
			& $\approx 1.099$
			& $\norm{A}_S/n$
			& $\norm{A}_S$ \\[2pt]
			$L_{yy}$
			& $1$
			& $x^2/4+\tfrac12 \rightarrow 4.5$
			& $1$
			& $\sigma^2/\lambda$ \\[2pt]
			$\mu$
			& $1$
			& $\tfrac12$
			& $\lambda_1 n^2 = 1$
			& $\sigma^2/\lambda$ \\[2pt]
			$\kappa_y = L_{yy}/\mu$
			& $1$ & $9$ & $1$ & $1$ \\[2pt]
			$\rho$
			& $2$
			& $4$
			& $\alpha\lambda_2/2 = 0.05$
			& $\tfrac{L}{L+1} < \tfrac13$ \\
			\bottomrule
		\end{tabular}
	\end{table}

	\paragraph{Algorithmic baselines} 
    The baselines for comparison vary depending on the applicable algorithms in each setting. In the toy example~\eqref{toy_example}, $\prox{\Phi(\cdot,y_k)}$
    admits a closed form, so we compare our methods against two alternatives. The first performs a \emph{proximal} step in $x$: the smooth version of \emph{Proximal Alternating Gradient Descent--Ascent} ({alt-GDA})~\cite{bot2023lternatingproximalgradientstepsstochastic}. The second performs a \emph{gradient-descent} step in $x$: the simultaneous version of \algoone\footnote{That is, \algoone with
    the update $y_{k+1}=\prox{\eta_y h}\big(y_k+\eta_y \nabla_y
    \Phi(x_k,y_k)\big)$ for $y$.}, which we call \emph{Parallel Proximal Gradient Descent--Ascent} ({PPGDA})~\cite{Cohen2025AlternatingandParallelProximalGradientMethods}. Since in the Dirac-GAN problem~\eqref{strongly_concave_dirac_gan} the iterates $(x_k)_k$ generated by our methods are restricted to a compact set, we compare \algoone and \algotwo with {AGP}~\cite{Xu2023projalgo} and {Smoothed-GDA}~\cite{Zhang2020Asingleloopsmothed}. We recall that when the iterates remain in the compact set (without projection) AGP is equivalent to \algoone. Since the RLR problem~\eqref{rlr_model} was studied in~\cite{lin2025two}, we compare our methods against the alternating version of smooth \emph{Two-Timescale Gradient Descent--Ascent} (alt-{TTGDA})\footnote{Note that the alternating counterpart of smooth TTGDA~\cite{lin2025two} is covered in~\cite[Sect. 4.2]{bot2023lternatingproximalgradientstepsstochastic} as a particular case of the algorithm proposed therein.}~\cite{bot2023lternatingproximalgradientstepsstochastic}.
	
	\paragraph{Computation of proximal operators}
	As previously motivated, \algotwo is mostly interesting when one has access to $\prox{\Phi(\cdot,y_k)}$ in closed form or via a neural network. For the sake of completeness, we still present results for \algotwo when the proximal operator is unavailable, and approximate it via an inner-loop with Newton's method (see Section~\ref{sec::additionalExp} for implementation details). In this case, we report the measured CPU-time to account for these extra computations.
	
	
		\begin{table}[ht]
		\centering
		\caption{\gt{Number of iterations and time to reach
			$\norm{\nabla\varphi(x_k)}\le\varepsilon$ and step-size values on $3$ problems: toy experiment, Dirac-GAN and  RLR. The step-sizes for the RLR problem~\eqref{rlr_model} are the same for all the instances so we reported them only once. Best results are displayed in bold.}}
		\label{tab:results}
		\setlength{\tabcolsep}{1pt}
		\begin{tabular}{@{}llccccccc@{}}
			\toprule
			\multicolumn{2}{@{}l}{}
			& \textbf{GD-RGA} \ \ & \textbf{PD-RGA} \ \ & PPGDA \ \ & \makecell{alt-\\GDA} \ & \makecell{alt-\\TTGDA} & AGP
			& \makecell{Smoothed-\\GDA} \\
			\midrule
			
			& Iterations & $\mathbf{16}$ & $53$ & $124$ & $260$ & --- & --- & --- \\
			\multirow{2}{*}{Toy } & Time (s)         & $\mathbf{3.52\mathrm{e}{-}05}$ & $6.77\mathrm{e}{-}05$
			& $0.0001$ & $0.0002$ & --- & --- & --- \\
			& $\eta_y$    & 1    & 1   & 1    & 0.38 & --- & --- & --- \\
			& $\eta_x$ & 0.29 & 0.2 & 0.06 & 0.03 & --- & --- & --- \\
			\vspace{0.01cm}\\ 
		
			\midrule
			& Iterations & $\mathbf{2255}$ & $2439$ & --- & --- & --- & $116420$ & $374008$ \\
			\multirow{2}{*}{D-GAN } & Time (s)         & $\mathbf{0.032}$ & $0.061$ & --- & --- & --- & $1.87$ & $6.21$ \\
				& $\eta_y$    & 0.22    & 0.22   & ---   & --- & --- & 0.006 & 0.005 \\
			& $\eta_x$ & 0.016 & 0.015 & --- & --- & --- & 3.1 $\times 10^{-4}$ & 0.048 \\
			\vspace{0.01cm}\\ 
			
			\midrule
			\multirow{2}{*}{\makecell[l]{D-GAN\\ with}}
			& Iterations & $\mathbf{2004}$ & $2188$ & --- & --- & --- & $116420$ & $374008$ \\
			& Time (s)         & $\mathbf{0.047}$ & $0.087$ & --- & --- & --- & $1.87$ & $6.21$ \\
			\makecell[l]{$h(y)=\tfrac{y^2}{4}$} \ & $\eta_y$    & 0.25    & 0.25   & ---   & --- & --- & 0.006 & 0.005 \\
			& $\eta_x$ & 0.018 & 0.016 & --- & --- & --- & 3.1 $\times 10^{-4}$ & 0.048 \\
			\vspace{0.01cm}\\ 
			
			\midrule
			\multirow{2}{*}{\makecell[l]{RLR\\ \texttt{a9a}}}
			& Iterations & $43$ & $\mathbf{14}$ & --- & --- & $2246$ & --- & --- \\
			& Time (s)         & $\mathbf{0.34}$ & $0.48$ & --- & --- & $12.85$ & --- & --- \\
			\vspace{0.01cm}\\ 
			
			\midrule
			\multirow{2}{*}{\makecell[l]{RLR\\ \texttt{mushrooms}}}
			& Iterations & $43$ & $\mathbf{13}$ & --- & --- & $2237$ & --- & --- \\
			& Time (s)         & $\mathbf{0.12}$ & $0.16$ & --- & --- & $5.31$ & --- & --- \\
			\vspace{0.01cm}\\ 
			\midrule
			\multirow{2}{*}{\makecell[l]{RLR\\ \texttt{news20}}}
			& Iterations & $45$ & $\mathbf{16}$ & --- & --- & $2479$ & --- & --- \\
			& Time (s)         & $\mathbf{6.8}$ & $14.36$ & --- & --- & $392$ & --- & --- \\
			& $\eta_y$ &  1 & 1 & --- & --- & 1 & --- & --- \\
			& $\eta_x$ & 4.8 & 19.8 & --- & --- & 0.1 & --- & --- \\
			\bottomrule
		\end{tabular}
	\end{table}
	
	
	\subsection{Results and discussion}
	
	\paragraph{Speed to reach $\varepsilon$-stationary point}
	
	As reported in Table~\ref{tab:results}, \algoone and \algotwo are always faster to reach $\norm{\nabla\varphi(x_k)}\le\varepsilon$ than the baseline algorithms on every problem. On the toy example,  Figure~\ref{result_toy_example} shows that all four methods converge to the same critical point $x^\star=-2/3$, just at different speeds. On the other hand, on the Dirac-GAN problem (Figure~\ref{result_dirac_gan}) and RLR\  (Figure~\ref{minimization_rlr_results}), the gap of convergence is between one and two orders of magnitude.
	Note that the decrease are not monotonous along the runs, which agree with the theoretical results that do not guarantee that $\norm{\nabla \varphi(x_k)}$ is non-increasing. 
	
	Finally, although the number of iterations required by \algotwo is the smallest of all methods on RLR, it is actually slightly slower than \algoone in wall-clock time due to the inner-loop required to compute $\prox{\Phi(\cdot,y_k)}$ on this problem. The interest of \algotwo lies instead in the PnP experiment presented later.

	\begin{figure}[ht]\label{pic_example_1}
		\centering
		\begin{minipage}{0.5\textwidth}
			\centering
			\includegraphics[width=1\linewidth]{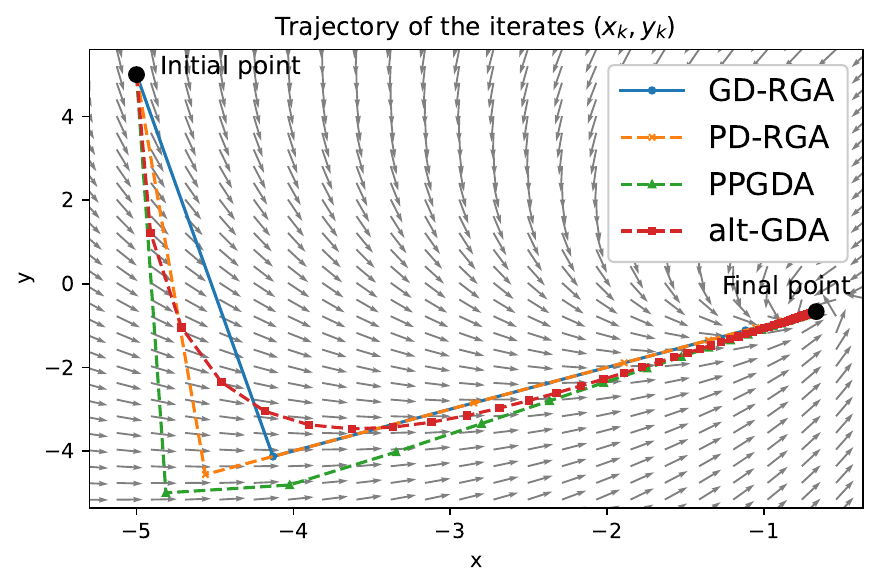}
		\end{minipage}\hfill
		\begin{minipage}{0.5\textwidth}
			\centering
			\includegraphics[width=1\linewidth]{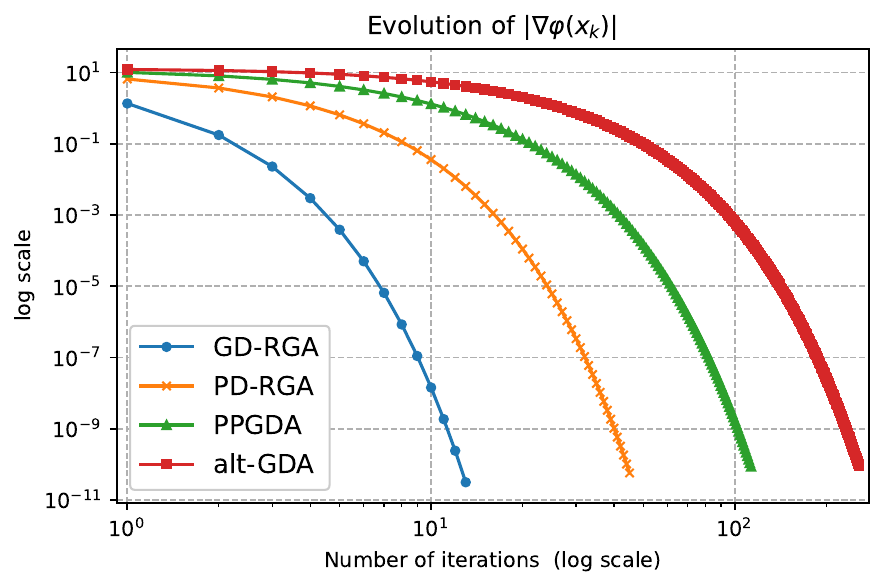}
		\end{minipage}
		\caption{\gt{Comparison of the methods on the toy example~\eqref{toy_example}, starting from the initial point $(x_0, y_0) = (-5, 5)$. \emph{Left:} trajectories of the iterates $(x_k, y_k)$. \emph{Right:} evolution of the convergence criterion $|\nabla \varphi(x_k)|$. All algorithms converge to $x^\star = -2/3$.}}
		\label{result_toy_example}
	\end{figure}

	\begin{figure}[ht]
		\centering
		\begin{minipage}{0.5\textwidth}
			\centering
			\includegraphics[width=.75\linewidth]{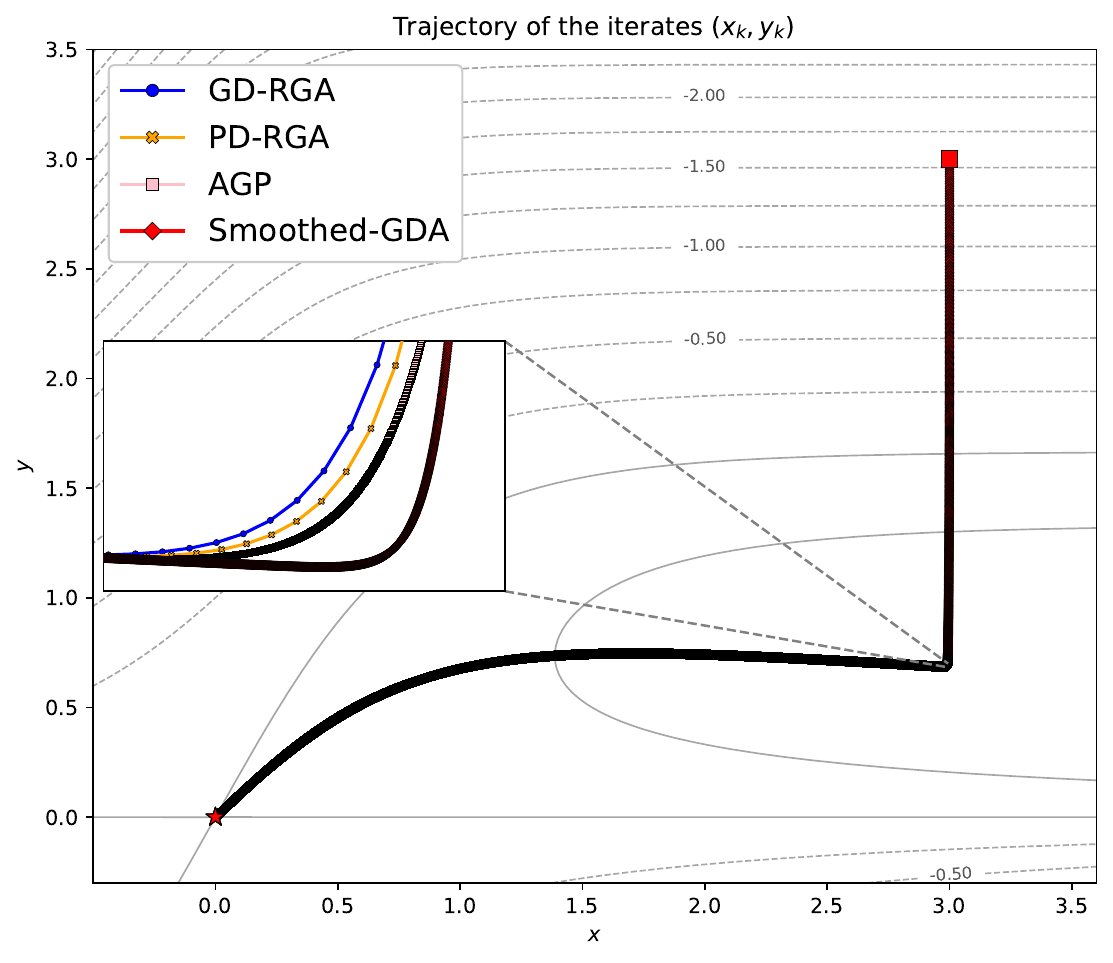}
		\end{minipage}\hfill
		\begin{minipage}{0.5\textwidth}
			\centering
			\includegraphics[width=1\linewidth]{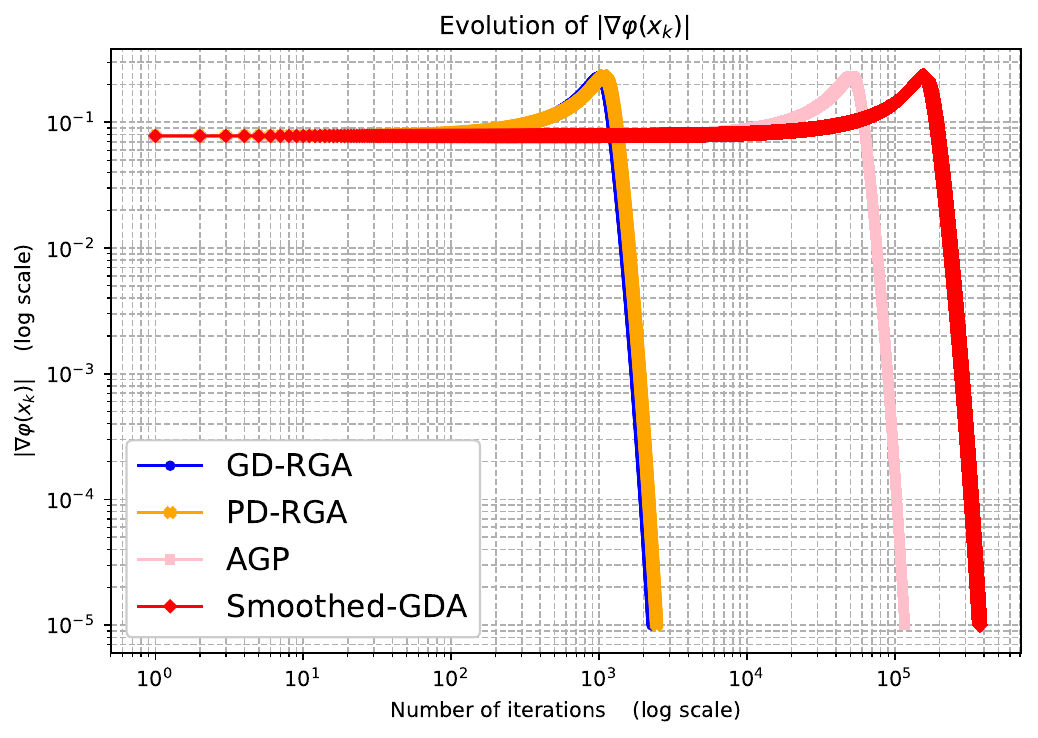}
		\end{minipage}
		\caption{\gt{Comparison of the methods on the Dirac-GAN example~\eqref{strongly_concave_dirac_gan}, starting from the initial point $(x_0, y_0) = (3, 3)$ and taking $z_0=x_0$ for {Smoothed-GDA}. \emph{Left:} Trajectories of the iterates $(x_k,y_k)$. \emph{Right:} Evolution of the convergence criterion $\lvert \nabla \varphi(x_k) \rvert$.}}
		\label{result_dirac_gan}
	\end{figure}

	\begin{figure}[ht]
		
		\begin{tabular}{ccc}
			\hspace{-0cm}\footnotesize
			a9a &\hspace{-0cm} \footnotesize mushrooms & \hspace{-0cm}\footnotesize news20 
			\\
			\includegraphics[width=0.31\linewidth]{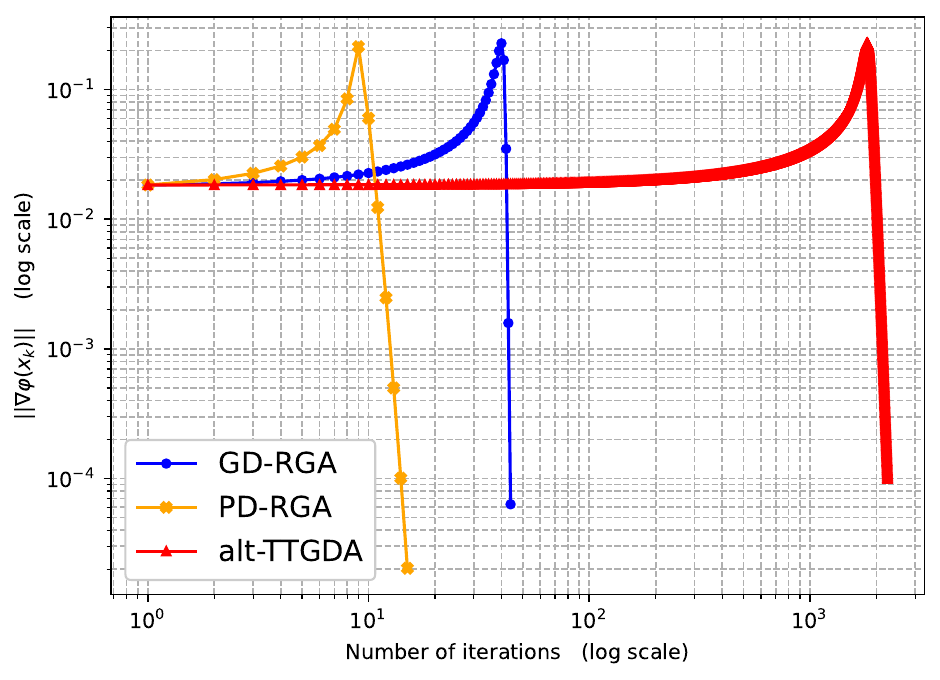}
			&
			\includegraphics[width=0.31\linewidth]{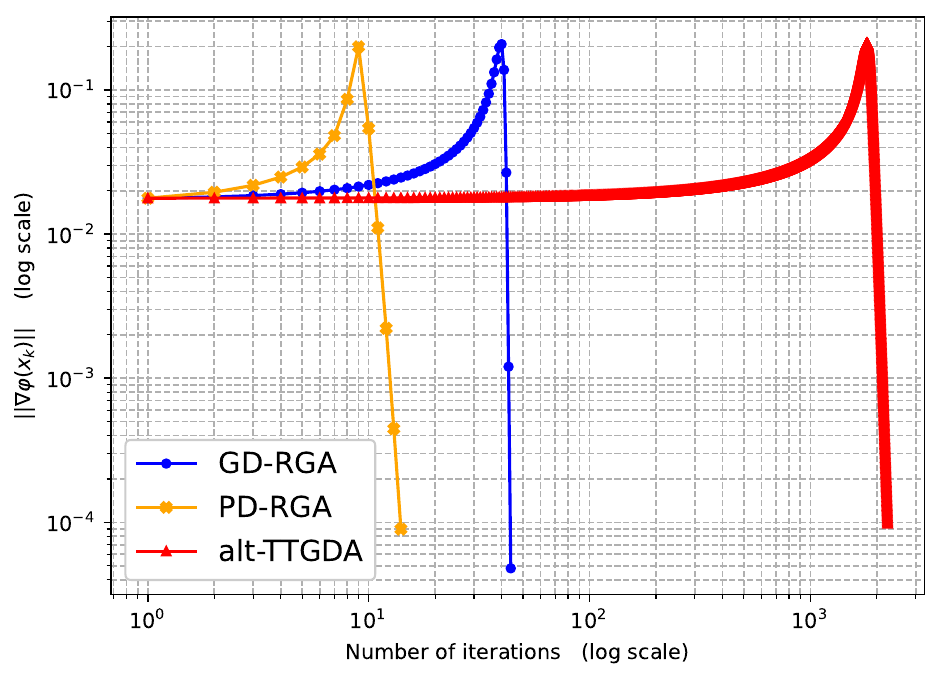}
			&
			\includegraphics[width=0.31\linewidth]{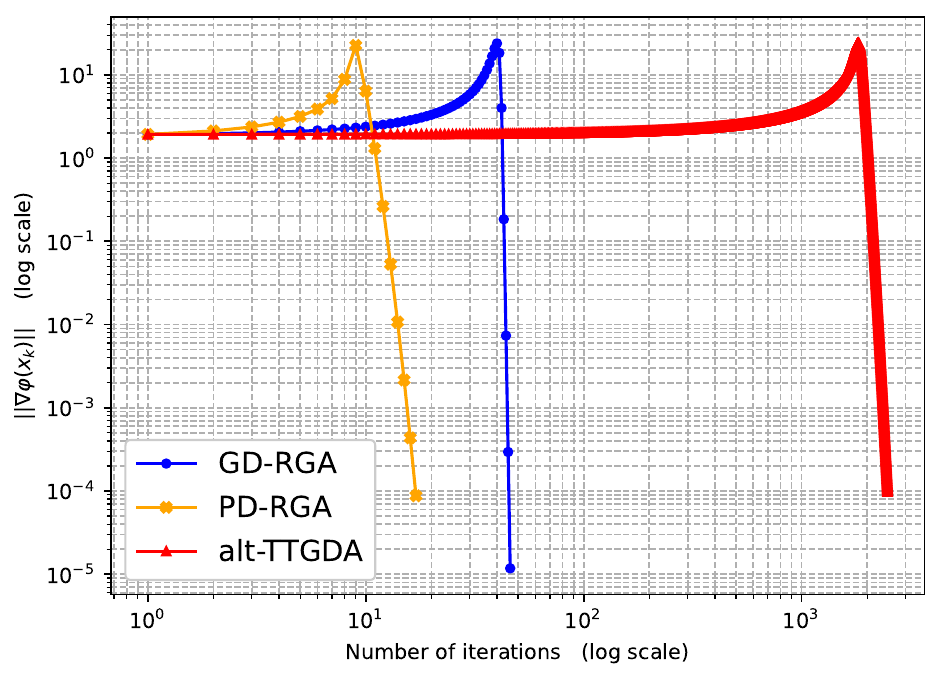}\\
		\end{tabular}
		
		\caption{\gt{Performance of all the algorithms with 3 LIBSVM datasets using $(x_0, y_0) = (\mathbf{1}_d, \mathbf{1}_n/n$) as initial point.}}
		\label{minimization_rlr_results}
	\end{figure}
	
	\paragraph{Gain of the improved step-size}
	
	As explained in the related work section, in some particular cases, algorithms in the literature (notably AGP and sometimes alt-TTGDA) differ from \algoone only by a projection step. When such projection is not triggered, the algorithms are equivalent. 
	This is the case with the Dirac-GAN problem, as evidenced by the trajectories in Figures~\ref{result_dirac_gan} and~\ref{line-search-dirac-gan}. We use this setting to compare the effect of our improved step-size with that of the literature. On this problem, AGP and \algoone coincide (when we interpret $h \equiv 0$), but our analysis allows larger theoretical step-sizes
	$\eta_x=0.016$ and $\eta_y=0.22$ against $\eta_x= 3.1 \times 10^{-4}$ and $\eta_y=0.006$, leading to the significantly faster convergence previously discussed (see Figure~\ref{result_dirac_gan}).  
	
	In the RLR experiment, we follow \cite{lin2025two} for the parameters of alt-TTGDA and choose $\eta_y$ on a grid $\{1,10^{-1}\}$ and then\footnote{This tuning is slightly less restrictive than the theoretical choice for alt-TTGDA studied in~\cite{bot2023lternatingproximalgradientstepsstochastic}. Taking the optimistic value $L_{\Phi}=1$, this choice gives $\eta_y \leq 1$ and $\eta_x \leq 1/12$. Recall that $L_{\Phi}$ denotes the joint Lipschitz constant of $\nabla\Phi$ defined in~\eqref{jointly_lipschitz_constant_def}.}
	$\eta_y/\eta_x\in\{10,10^2,10^3\}$, so that $\eta_x<\eta_y$ is imposed by construction. Note that the same constrained\footnote{Note that neither our analysis nor~\cite{Cohen2025AlternatingandParallelProximalGradientMethods} imposes the constraint $\eta_x < \eta_y$; instead, their ratio is allowed to adapt to each problem.} parameterization appears in \cite{bot2023lternatingproximalgradientstepsstochastic}. 
	This RLR experiment leads to the same conclusion as the Dirac-GAN one: although \algoone and alt-TTGDA share the same update rules, their admissible step-size ranges differ, and \algoone converges faster, demonstrating the advantage of a larger theoretical step-size.

	To further check this advantage, we consider the Dirac-GAN problem for different values of $\eta_x$. Recall that the bound w.r.t. $y$ in AGP (\cite[Thm. 3.1]{Xu2023projalgo}) is
	$\eta_y\le\mu/(4L_{yy}^2)=1/(4\kappa_y L_{yy})$, while our analysis permits $\eta_y\le 1/L_{yy}$, which is larger since $1/(4\kappa_y L_{yy})\le 1/L_{yy}$ for any	$\kappa_y\ge 1/4$. 
	For both the theoretical step-size $\eta_y=1/(4\kappa_y L_{yy})$ and the non-admissible step-size $\eta_y=1/L_{yy}$, we then tuned $\eta_x$ in AGP by grid search. The corresponding results, provided in Figure~\ref{line-search-dirac-gan}, show that properly tuned AGP, despite violating the theoretical step-size bound with respect to $y$, becomes competitive against \algoone, highlighting again the importance of step-size for performance.  Note that we exclude Smoothed-GDA from this particular instance because its theoretical choice of $\eta_y$ depends on $\eta_x$, which prevents us from varying $\eta_x$ while keeping $\eta_y$ fixed.

	\begin{figure}[ht]
		\centering
		\begin{minipage}{0.5\textwidth}
			\centering
			\includegraphics[width=.75\linewidth]{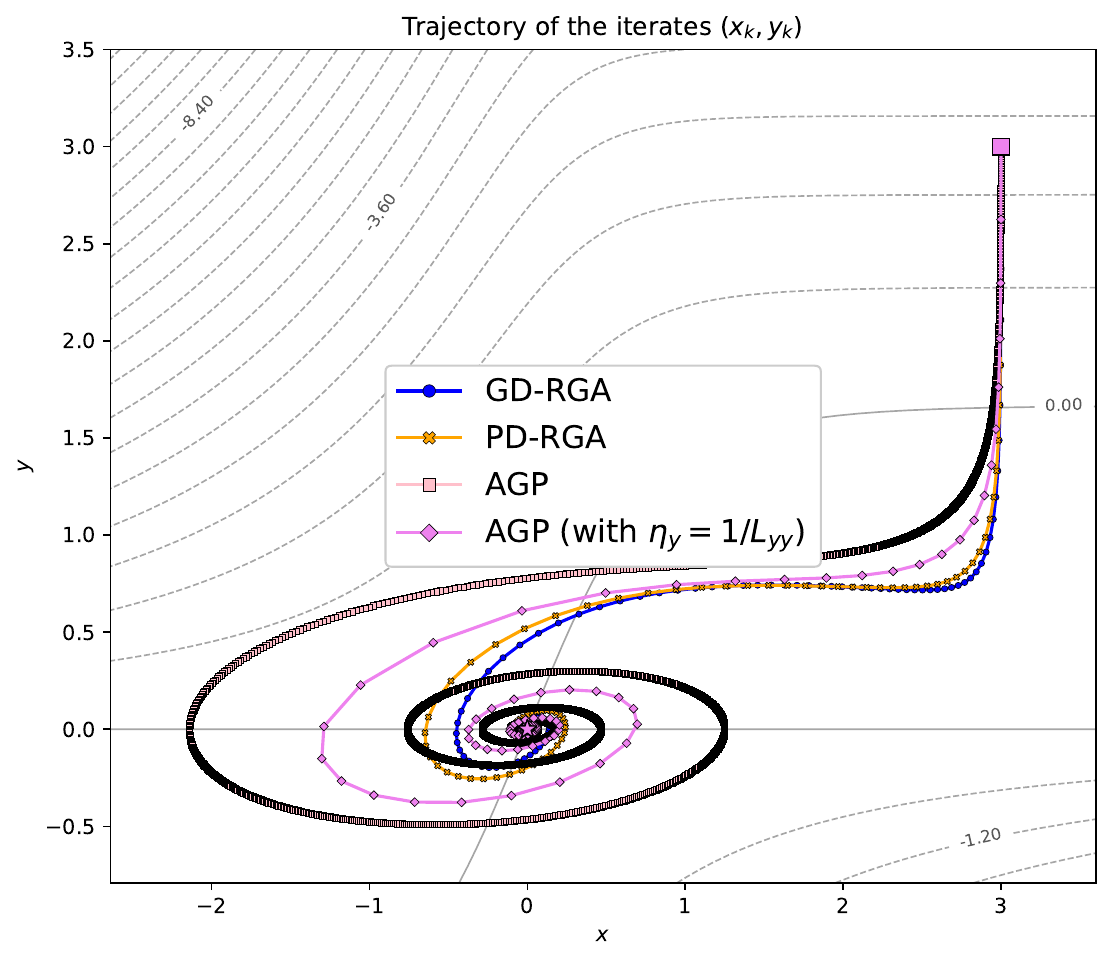}
		\end{minipage}\hfill
		\begin{minipage}{0.5\textwidth}
			\centering
			\includegraphics[width=1\linewidth]{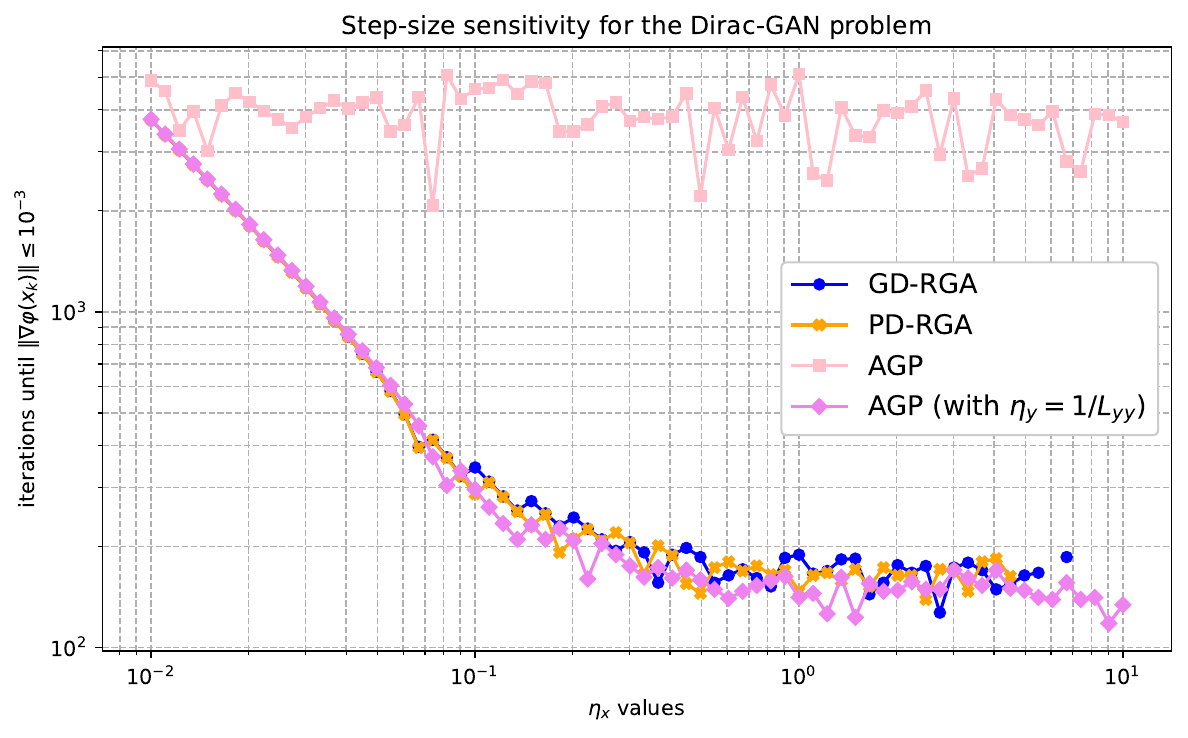}
		\end{minipage}
		\caption{\gt{Effect of the choice of the step-size $\eta_x$. \emph{Right:} For \algoone and \algotwo the step-size w.r.t. y is fixed to $1/L_{yy}$, whereas \textbf{AGP} is run using both the theoretical step-size $\mu/(4L_{yy}^2)$ (pink curve) from \cite{Xu2023projalgo} and $1/L_{yy}$ (purple curve) that violates this theoretical constraint but leads to faster convergence.  \emph{Left:} Trajectory of the iterates $(x_k,y_k)$ with the initial point $(x_0,y_0)=(3,3)$ and the best $\eta_x$ chosen via grid search.}}
		\label{line-search-dirac-gan}
	\end{figure}

	\paragraph{The effect of the regularization formulation}
	We now compare the two formulations of the Dirac-GAN problem introduced in
	Section~\ref{sec:setting}. Recall that carrying the quadratic term in $h$ rather than in the
	coupling leaves the constants of Table~\ref{tab:lipschitz} unchanged, except for
	$\tilde L_{yy}=4$ in place of $4.5$. Hence $\tilde\kappa_y=8$ in place of $9$, which leads to larger step-sizes; recall that $\tilde L_{yy}$ and $\tilde\kappa_y$ are associated with $\tilde\Phi$.
	As displayed in Table~\ref{tab:results} (third block), this reformulation slightly reduces the number of iterations. As the trajectories produced are indistinguishable from those in Figure~\ref{result_dirac_gan}, we omit them.

	\paragraph{Application to Plug-and-Play} 
	We now consider two Plug-and-Play image restoration problems, where the function $g$ in \eqref{min_max_formulation_pnp} has no closed form but its proximal operator
	is parameterized by a neural network; this is the main motivation for studying \algotwo. Among the methods considered that perform proximal steps in $x$, \algotwo achieves the best performance, which is why we select it for this experiment. Note that to avoid retraining the neural network, \algotwo is always run with step-size $\eta_x=1$, as explained in Appendix~\ref{sec::additionalExp}. 
	
	}

	\textbf{Super-resolution} \
	In single image super-resolution, a low-resolution image $b \in \R^n$ is observed from the unknown high-resolution one $x \in \R^d$ via $b = SHx + \gt{\xi}$ where $H \in \R^{d \times d}$ is the convolution with anti-aliasing kernel and $S \in \R^{n \times d}$ is the standard $s$-fold downsampling, with $s=2$ and $d=s^2 \times n$, so that $A=SH$~\cite[Sect. ~5.2.2]{Hurault2022ProximalDF} and \cite[App. ~F.4]{renaud2024plugandplay}. 
	Here $\norm{A}_S=0.24$, and we consider $\sigma=0.03$ so we set \gt{$\lambda=0.0032$} and $\gt{\eta_y = 3.55}$. As shown in Figure~\ref{sr-start-fish} the method provides high quality restoration, with PSNR values on par with state-of-the-art Plug-and-Play methods~\cite[Sect.~5.2.2]{Hurault2022ProximalDF}. 
	
	\begin{figure}[ht]
		\centering
		
		\begin{tabular}{cccc}
			\hspace{-0.3cm}\footnotesize
			Ground truth &\hspace{-0.3cm} \footnotesize Observation & \hspace{-0.3cm}\footnotesize Reconstruction &\hspace{-0.3cm} \footnotesize Evolution of $\Vert \nabla \varphi(x_k)\Vert$
			\\
			\hspace{-0.3cm}\includegraphics[width=0.23\linewidth]{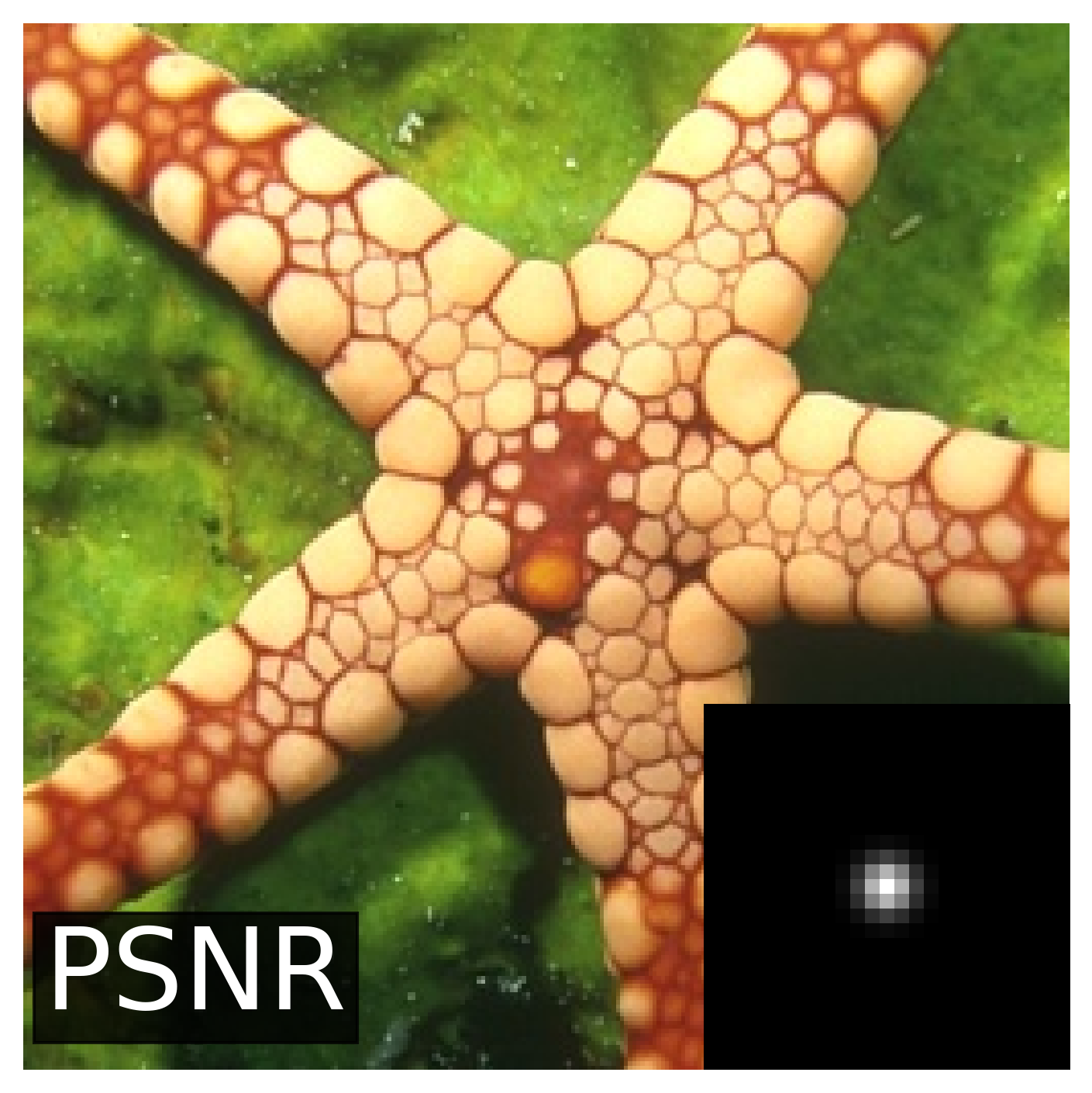}
			&\hspace{-0.3cm}
			\includegraphics[width=0.23\linewidth]{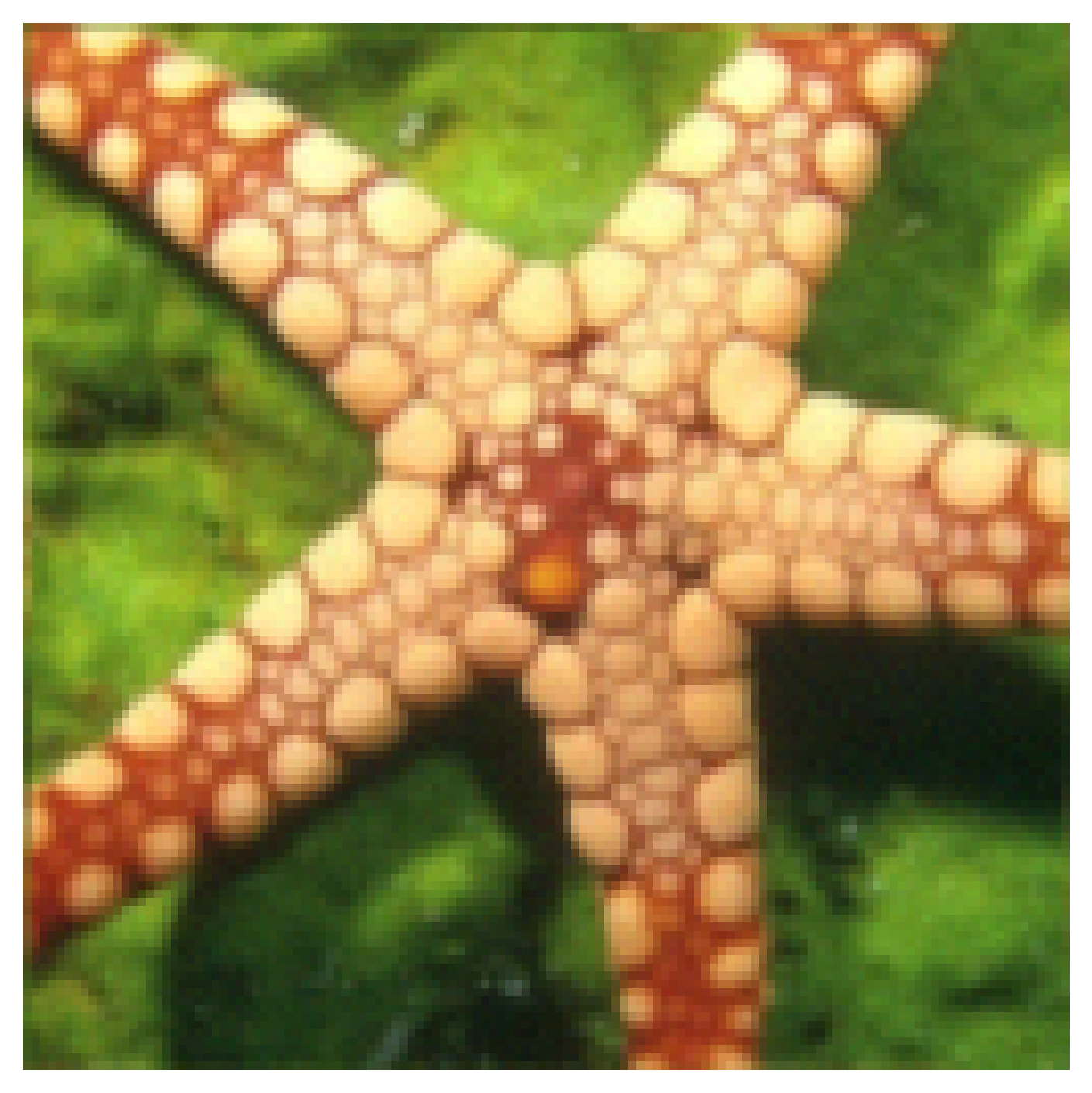}
			& \hspace{-0.3cm}
			\includegraphics[width=0.23\linewidth]{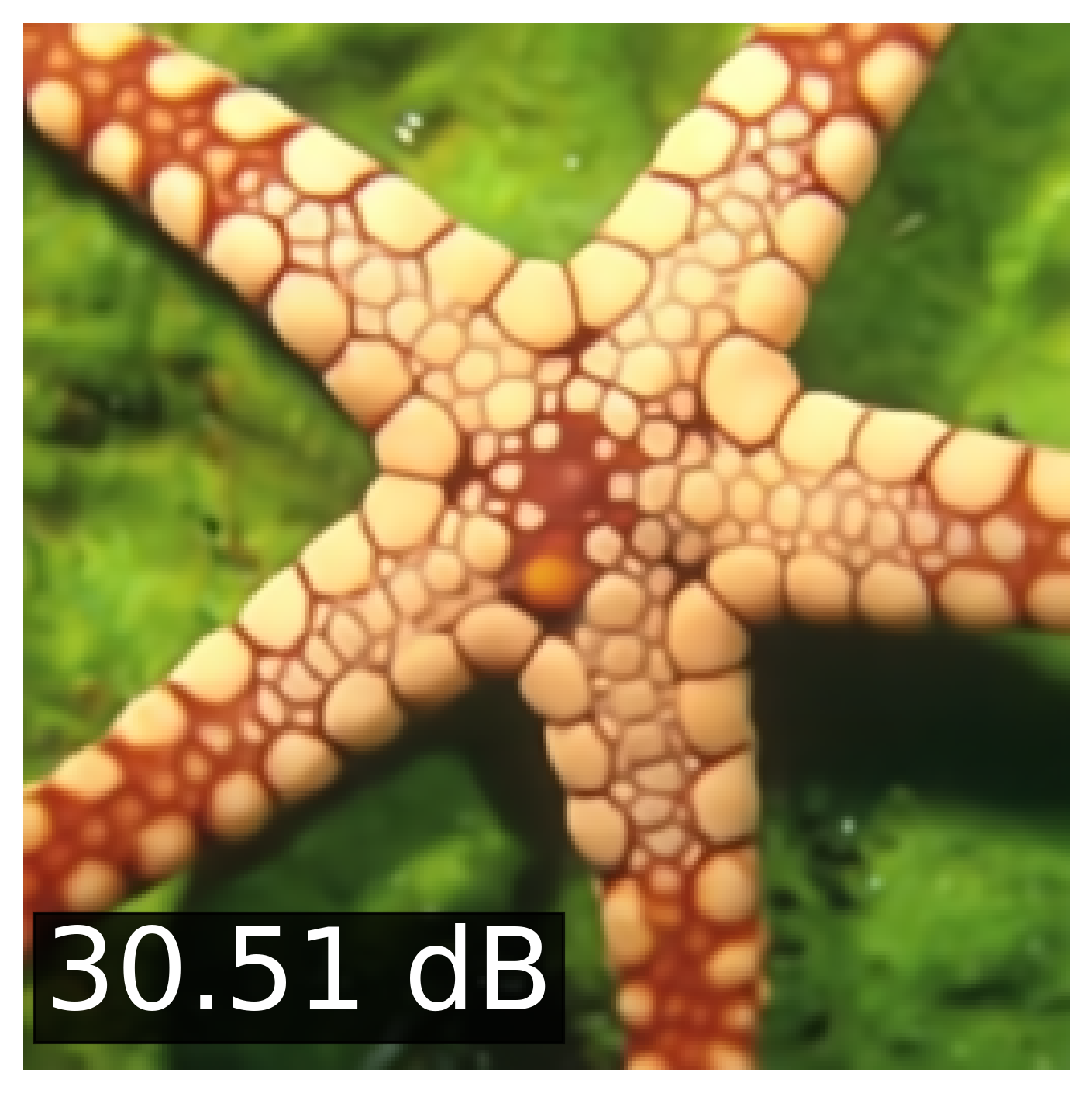}
			&\hspace{-0.3cm}
			\includegraphics[width=0.23\linewidth]{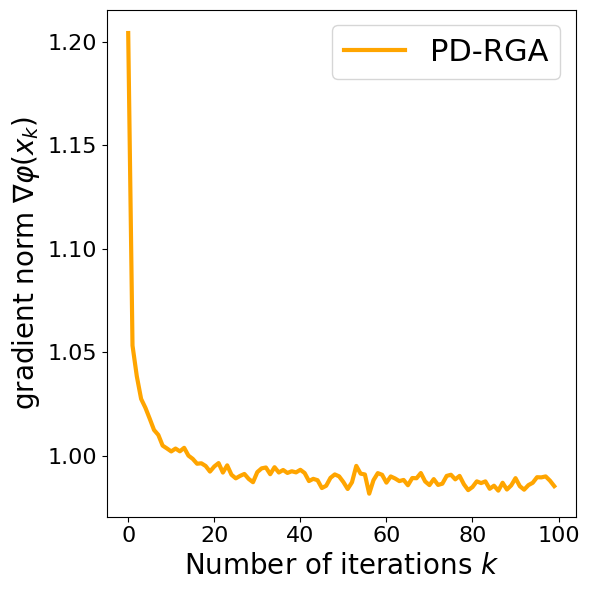}
		\end{tabular}
		
		\caption{Super-resolution of a starfish from the data set Set3C downscaled with the indicated blur kernel and scale $s=2$.}
		\label{sr-start-fish}
	\end{figure}

	\textbf{Deblurring} \ For image deblurring, the operator $A$ is a convolution  with circular boundary conditions. Specifically, $A=F \Lambda F^*$ where $F$ is the orthogonal matrix of the discrete Fourier transform ($F^*$ being its inverse) and $\Lambda$ a diagonal matrix so $b = F \Lambda F^* x + \gt{\xi}$~\cite[Sect. 5.1.1]{Hurault2022ProximalDF}.
	Here $\norm{A}_S=1$ and we consider $\sigma=0.03$ so we set $\gt{\lambda=0.00018}$ and $\gt{\eta_y = 0.2}$. \gt{Figure~\ref{db-butter-fly}} illustrates that  the image reconstructed by \algotwo improves the PSNR by nearly 10 dB, with similar results as reported in~\cite[Sect.~5.2.2]{Hurault2022ProximalDF}.

	\begin{figure}[ht]
		\centering
		
	\begin{tabular}{cccc}
		\hspace{-0.3cm}\footnotesize
		Ground truth &\hspace{-0.3cm} \footnotesize Observation &\hspace{-0.3cm} \footnotesize Reconstructed &\hspace{-0.3cm} \footnotesize Evolution of $\Vert \nabla \varphi(x_k)\Vert$
		\\
		\hspace{-0.3cm}\includegraphics[width=0.23\linewidth]{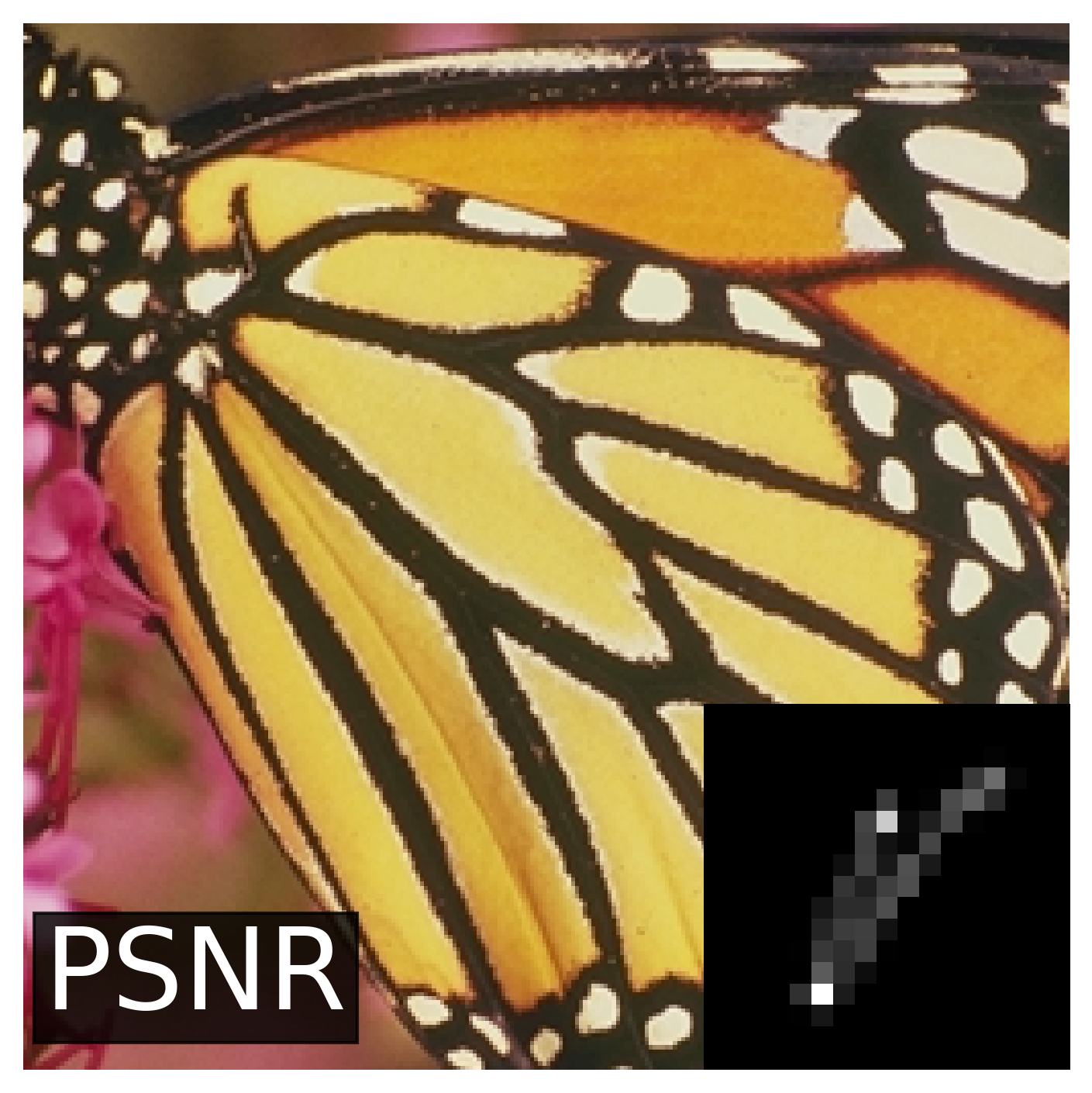}
		& \hspace{-0.3cm}
		\includegraphics[width=0.23\linewidth]{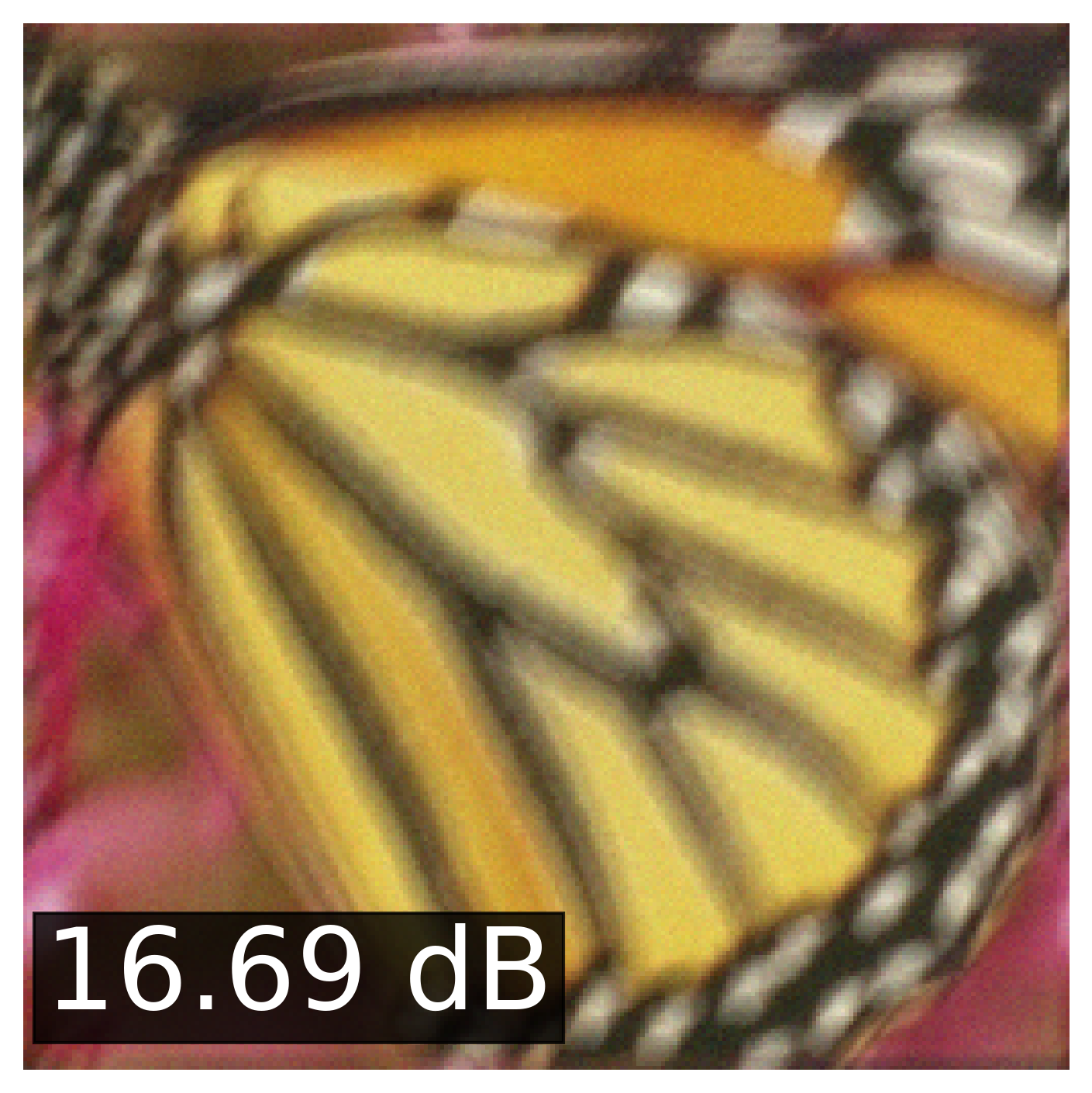}
		& \hspace{-0.3cm}
		\includegraphics[width=0.23\linewidth]{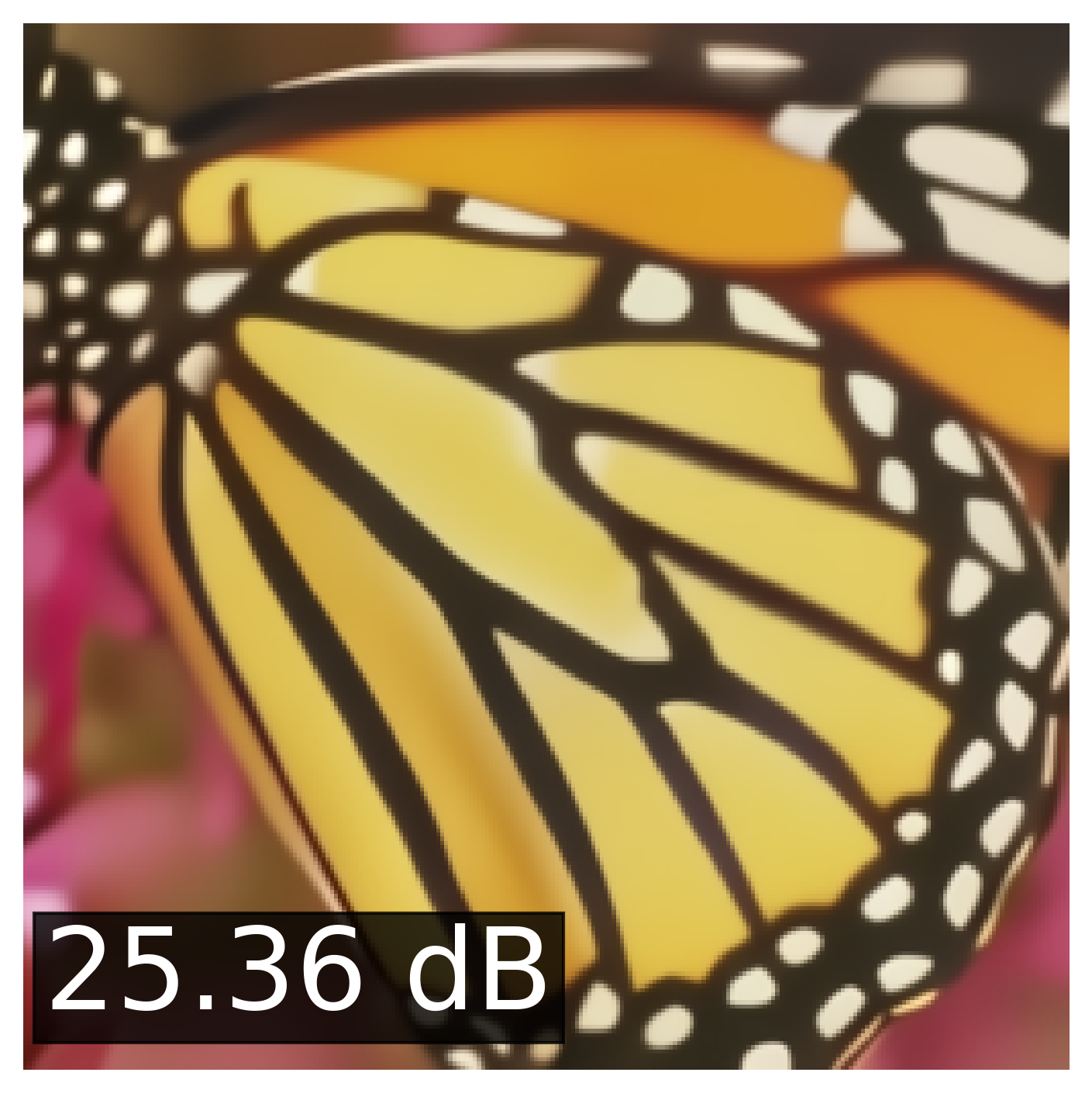}
		&\hspace{-0.3cm}
		\includegraphics[width=0.23\linewidth]{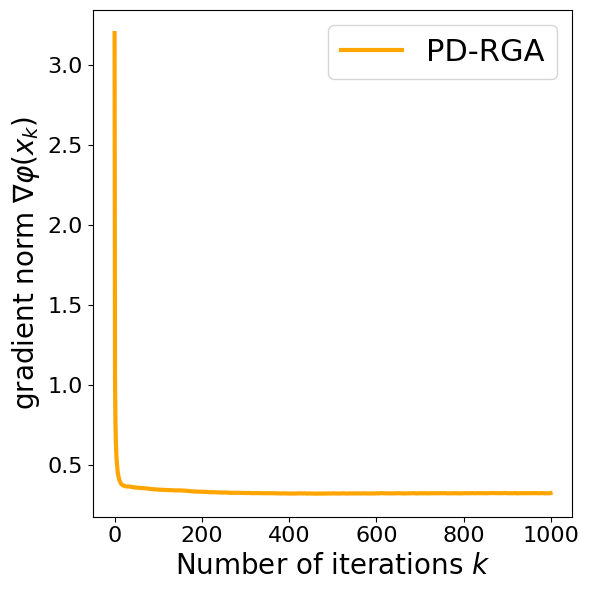}
	\end{tabular}
	\caption{Deblurring of \gt{a butterfly} from the data set Set3C.}
	\label{db-butter-fly}
	\end{figure}


\section{Conclusions and perspectives}
In this work, we studied gradient algorithms for solving nonconvex-strongly concave min-max problems. We presented a direct convergence analysis for the gradient descent-ascent algorithm, improving upon state-of-the-art step-size rules. Leveraging our general proof strategy, we proposed the first convergence analysis for \algotwo in a weakly convex–strongly concave setting. Through numerical examples, we demonstrated that \algotwo allows to tackle regularizing inverse problems where neural networks act as proximal operators.

From an application standpoint, \algotwo opens the possibility of extending the applications considered in~\cite{bot2022acceleratedminimaxalgorithmconvexconcave}, such as multi-kernel support vector machines and classification with minimax group fairness, to nonconvex scenarios. Future work will focus on extending our proof strategy to nonsmooth couplings, \gt{non-Euclidean} scenarios (for example using Bregman proximal operators~\cite{Cohen2025AlternatingandParallelProximalGradientMethods}), and stochastic variants of our algorithms as well as on relaxing the strong concavity assumption.

\appendix
	
	\section{Proofs for the analysis of convergence of \algoone}\label{missing_proofs_of_algo_1}
	
	\subsection{Proof of Lemma~\ref{Algo1-lemma-1}}\label{Proof_lemma_1_algo_1}
	
	\begin{proof}
		By the descent Lemma~\ref{Descen_lemma} we get
		\begin{equation*}
			\varphi(x_{k+1}) \leq \varphi(x_k) + \langle \nabla \varphi(x_k), x_{k+1} - x_k \rangle + \frac{\lc}{2} \norm{x_{k+1} - x_k}^2.
		\end{equation*}
		By \eqref{algo_1_update_x_sect_3} we have $x_{k+1} - x_k = - \eta_x \nabla_x \Phi(x_k, y_k)$, so it implies
		\begin{equation}\label{Algo_1_lemma_1_eq0}
			\varphi(x_{k+1}) \leq \varphi(x_k) - \eta_x \left\langle \nabla \varphi(x_k), \nabla_x \Phi(x_k, y_k) \right\rangle + \frac{\lc\eta_x^2}{2} \norm{ \nabla_x \Phi(x_k, y_k) }^2.
		\end{equation}
		\gt{
			Using the polarization identity we have
			\begin{equation}\label{Algo_1_lemma_1_eq1}
				-2 \langle \nabla \varphi(x_k), \nabla_x \Phi(x_k,y_k) \rangle = - \norm{\nabla \varphi(x_k)}^2 - \norm{\nabla_x \Phi(x_k,y_k)}^2 + \norm{\nabla \varphi(x_k) -\nabla_x \Phi(x_k,y_k)}^2.
			\end{equation}
			Plugging \eqref{Algo_1_lemma_1_eq1} in \eqref{Algo_1_lemma_1_eq0} we get
			\begin{align}\label{Algo_1_lemma_1_eq3}
				\varphi(x_{k+1}) \leq \varphi(x_k) - \frac{\eta_x}{2}\norm{\nabla \varphi(x_k)}^2 - \frac{\eta_x}{2}(1- L_{\varphi} \eta_x)\norm{\nabla_x \Phi(x_k,y_k)}^2 \notag \\
				+ \frac{\eta_x}{2}\norm{\nabla \varphi(x_k) - \nabla_x \Phi(x_k,y_k)}^2.
			\end{align}
		}
		Then by Proposition \ref{Smoothness_of_the_max_function}, we obtain
		\begin{align}\label{Algo_1_lemma_1_eq4}
			\norm{ \nabla \varphi(x_k) - \nabla_x \Phi(x_k, y_k) }^2 &= \norm{ \nabla_x \Phi(x_k, y_k^*) - \nabla_x \Phi(x_k, y_k) }^2 \leq L_{xy}^2 \norm{ y^*(x_k) - y_k }^2.
		\end{align}
		\gt{Plugging \eqref{Algo_1_lemma_1_eq4} into \eqref{Algo_1_lemma_1_eq3} and using that $\eta_x L_{\varphi} < 1$ by assumption, we deduce
			\begin{equation*}
				\varphi(x_{k+1}) \leq \varphi(x_k) - \frac{\eta_x}{2} \norm{\nabla \varphi(x_k)}^2 + \frac{\eta_x}{2} L_{xy}^2 \norm{ y^*(x_k) - y_k }^2.
			\end{equation*}
		}
		Summing over $k=0,\dots,N-1$ we deduce the result\footnote{\gt{Note that we do not drop the term $- \frac{\eta_x}{2} \norm{\nabla \varphi(x_k)}^2$, since it will be used later to obtain an $\varepsilon$-stationary point.}}.
	\end{proof}

	\subsection{Proof of Lemma~\ref{Algo1-lemma-2}}
	
	\begin{proof}
		We distinguish two cases based on the source of the $\mu$-strong concavity: whether it arises from $\Phi$ or from $-h$. \\
		\textbf{Case 1:} We assume that $\Phi(x,\cdot)$ is $\mu$-strongly concave in its second component for all $x \in \R^d$. We recall that $\Phi(x_{k+1},\cdot)$ is concave and $L_{yy}$-smooth (Assumption~\ref{assumption_3}), so that in this case, we necessarily have $\mu\leq L_{yy}$ and $\kappa_y=L_{yy}/\mu\geq 1$.
		From the ascent Lemma~\ref{Descen_lemma}, we first have
		\begin{multline}\label{general_bound_deltak_eq_1}
			\Phi(x_{k+1},y_{k+1}) - h(y_{k+1}) \geq \\
			\Phi(x_{k+1},y_k) + \langle \nabla_y \Phi(x_{k+1},y_k), y_{k+1}-y_k \rangle - \frac{L_{yy}}{2}\norm{y_{k+1}-y_k}^2 - h(y_{k+1}).
		\end{multline}
		Next, by optimality condition of the $\operatorname{prox}$ operator in~\eqref{algorithm_dproxa_b_sect_3}, there exists $s_{k+1} \in \partial h(y_{k+1})$ such that $0 =  s_{k+1} + \eta_y^{-1}(y_{k+1}-y_k - \eta_y \nabla_y\Phi(x_{k+1},y_k))$. Then, by definition of the subdifferential we have $h(y^*_{k+1}) \geq h(y_{k+1}) + \langle s_{k+1},y^*_{k+1}-y_{k+1} \rangle$. Plugging this relation into \eqref{general_bound_deltak_eq_1} we obtain
		\begin{align}
			\Phi(x_{k+1},y_{k+1}) - h(y_{k+1}) \geq&\;  \Phi(x_{k+1},y_k) + \langle \nabla_y \Phi(x_{k+1},y_k), y_{k+1}-y_k \rangle - \frac{L_{yy}}{2}\norm{y_{k+1}-y_k}^2 \notag
			\\
			- h(y_{k+1}^*) 
			&- \frac{1}{\eta_y}\langle y_{k+1}-y_k , y^*_{k+1}-y_{k+1} \rangle + \langle \nabla_y \Phi(x_{k+1},y_k), y^*_{k+1}-y_{k+1} \rangle. \label{general_bound_deltak_eq_1_5}
		\end{align} 
		Using again strong-concavity,  $\Phi(x_{k+1},y^*_{k+1}) \leq \Phi(x_{k+1},y_k) + \langle \nabla_y \Phi(x_{k+1},y_k), y^*_{k+1}-y_k \rangle - \mu/2 \norm{y_{k+1}^*-y_k}^2$, from which we bound~\eqref{general_bound_deltak_eq_1_5} by below obtaining
		\begin{align*}
			\Phi(x_{k+1},y_{k+1}) - h(y_{k+1}) \geq & \ \Phi(x_{k+1},y^*_{k+1}) - h(y^*_{k+1}) - \frac{L_{yy}}{2}\norm{y_{k+1}-y_k}^2\\
			&+ \frac{\mu}{2}\norm{y^*_{k+1}-y_k}^2 - \frac{1}{\eta_y} \langle y_{k+1}-y_k, y^*_{k+1} - y_{k+1} \rangle.
		\end{align*}
		As  $\langle y_{k+1} - y_k, y^*_{k+1} - y_{k+1} \rangle = \frac12\left( \norm{y_k - y^*_{k+1}}^2 - \norm{y_{k+1}-y^*_{k+1}}^2 - \norm{y_k - y_{k+1}}^2 \right)$, we get 
		\begin{align*}
			\Phi(x_{k+1},y_{k+1}) - h(y_{k+1}) \geq& \Phi(x_{k+1},y^*_{k+1}) - h(y^*_{k+1}) - \frac{1}{2}\left(L_{yy} - \frac{1}{\eta_y} \right)\norm{y_{k+1}-y_k}^2 \\
			&-  \frac{1}{2}\left( \frac{1}{\eta_y} - \mu\right)\norm{y^*_{k+1}-y_k}^2  
			+ \frac{1}{2\eta_y}\norm{y_{k+1}-y^*_{k+1}}^2.
		\end{align*}
		Multiplying by $2\eta_y$ and rearranging the inequality we obtain
		\begin{align}
			\norm{y_{k+1}-y^*_{k+1}}^2
			\le&
			\left(\eta_y L_{yy} - 1\right)\norm{y_{k+1}-y_k}^2
			+ \left(1 - \eta_y \mu\right)\norm{y^*_{k+1}-y_k}^2 \notag \\
			&+2\eta_y \left( \Phi(x_{k+1},y_{k+1}) -h(y_{k+1}) - \Phi(x_{k+1},y^*_{k+1}) + h(y^*_{k+1}) \right). \label{general_bound_deltak_eq_2_0}
		\end{align}
		Recalling that $\eta_yL_{yy}=\tau\in (0;1]$ and $\kappa_y\geq 1$, we have $\eta_y\mu = \tau/\kappa_y \leq 1$ and~\eqref{general_bound_deltak_eq_2_0} becomes
		\begin{align}
			\norm{y_{k+1}-y^*_{k+1}}^2
			\le& \left(1 - \tau/\kappa_y\right)\norm{y^*_{k+1}-y_k}^2 \notag \\
			&+ 2\tau/L_{yy} \left( \Phi(x_{k+1},y_{k+1}) -h(y_{k+1}) - \Phi(x_{k+1},y^*_{k+1}) + h(y^*_{k+1}) \right). \label{general_bound_deltak_eq_2}
		\end{align}
		The $\mu-$strong concavity of the coupling $\Phi$ w.r.t. $y$-variable implies that
		\begin{equation}\label{general_bound_deltak_eq_3}
			\Phi(x_{k+1},y_{k+1}) \leq \Phi(x_{k+1},y^*_{k+1}) + \langle \nabla_y \Phi(x_{k+1},y^*_{k+1}), y_{k+1}-y^*_{k+1} \rangle - \frac{\mu}{2}\norm{y_{k+1}-y^*_{k+1}}^2.
		\end{equation}
		From Lemma~\ref{Lipschitz_continuity_of_the_solution_mapping} we get $0 \in \nabla_y \Phi(x_{k+1},y^*_{k+1}) - \partial h(y^*_{k+1})$. We thus have $\nabla_y \Phi(x_{k+1},y^*_{k+1})\in \partial h(y^*_{k+1})$ so that
		\begin{equation}\label{general_bound_deltak_eq_4}
			-h(y_{k+1}) \leq -h(y^*_{k+1}) - \langle \nabla_y \Phi(x_{k+1},y^*_{k+1}), y_{k+1}-y^*_{k+1} \rangle.
		\end{equation}
		Summing \eqref{general_bound_deltak_eq_3} and \eqref{general_bound_deltak_eq_4} then gives
		\begin{equation*}
			\Phi(x_{k+1},y_{k+1})-h(y_{k+1}) \leq \Phi(x_{k+1},y^*_{k+1})-h(y^*_{k+1}) - \frac{\mu}{2}\norm{y_{k+1}^*-y_{k+1}}^2.
		\end{equation*}
		Replacing the last inequality in \eqref{general_bound_deltak_eq_2}, we get 
		\begin{align}
			\norm{y_{k+1}-y^*_{k+1}}^2 &\leq \left(1-\frac{\tau\mu}{L_{yy}}\right)\norm{y_{k+1}^*-y_k}^2 - \frac{\tau\mu}{L_{yy}}\norm{y^*_{k+1}-y_{k+1}}^2, \notag\\
			\left(1+\frac{\tau\mu}{L_{yy}}\right) \norm{y_{k+1}-y^*_{k+1}}^2 &\leq \left(1-\frac{\tau\mu}{L_{yy}}\right)\norm{y^*_{k+1}-y_k}^2 .\notag
		\end{align}
		Recalling that $\kappa_y=L_{yy}/\mu$, we obtain     
		\begin{align}         
			\delta_{k+1}=\norm{y_{k+1}-y^*_{k+1}}^2&\leq\left( \frac{\kappa_y-\tau}{\kappa_y+\tau}\right)\norm{y^*_{k+1}-y_k}^2 \leq \left( \frac{\kappa_y}{\kappa_y+\tau}\right)^2\norm{y^*_{k+1}-y_k}^2.
			\label{general_bound_deltak_eq_5}
		\end{align}
		\textbf{Case 2:} We now assume that $h$ is strongly convex. Our strategy in this case is to exploit the fact that the operator $\operatorname{prox}_h$  is a contraction. Also note that $\Id + \eta_y \nabla_y \Phi(z,\cdot)$ for $z \in \R^d$ is non-expansive, \textit{i.e.},
		\begin{equation}\label{chp_3_ascent_op_eq_1}
			\norm{\left(\Id + \eta_y \nabla_y \Phi(z,\cdot) \right)x - \left(\Id+ \eta_y \nabla_y \Phi(z,\cdot) \right)y}^2 \leq \norm{x-y}^2.
		\end{equation}
		Using \cite[Lem. 5.5, pp. 13]{renaud2025moreauenvelope} we have that $\operatorname{prox}_{\eta_y h}$ is $1/(1+\mu\eta_y)$-Lipschitz and by the fact that the theoretical maximizer $y^*_{k+1}$ is a fixed point of~\eqref{algorithm_dproxa_b_sect_3}. We have that 
		\begin{align}\label{chp_3_ascent_op_eq_1_5}
			\norm{y^*_{k+1}-y_{k+1}}^2 &= \norm{\prox{\eta_y h}(y^*_{k+1} + \eta_y \nabla_y \Phi(x_{k+1},y^*_{k+1})) - \prox{\eta_y h}(y_{k} + \eta_y \nabla_y \Phi(x_{k+1},y_{k}))}^2  \notag\\ 
			&\leq \left(\frac{1}{1 + \mu\eta_y} \right)^2\norm{y^*_{k+1}-y_{k} + \eta_y \left( \nabla_y \Phi(x_{k+1},y^*_{k+1})-\nabla_y \Phi(x_{k+1},y_{k})\right)}^2 \notag \\
			&\underset{\leq}{\eqref{chp_3_ascent_op_eq_1}} \left(\frac{1}{1 + \mu\eta_y} \right)^2 \norm{y^*_{k+1}-y_{k}}^2.
		\end{align}
		Since $\delta_{k+1}=\norm{y^*_{k+1}-y_{k+1}}^2$, $\eta_y=\tau/L_{yy}$, with $0<\tau \leq 1$, and $\kappa_y=L_{yy}/\mu$,~\eqref{chp_3_ascent_op_eq_1_5} becomes
		\begin{equation}\label{chp_2_lemma_33_eq_2}
			\delta_{k+1} \leq \left( \frac{\kappa_y}{\kappa_y + \tau}\right)^2 \norm{y^*_{k+1}-y_k}^2.
		\end{equation}
		
		\textbf{Conclusion} Both cases yield the same control (relations~\eqref{general_bound_deltak_eq_5} and \eqref{chp_2_lemma_33_eq_2}) for $\delta_k$, so it does not matter whether the strong concavity originates from $\Phi$ or $-h$.  Now using triangular and  Young inequalities in \eqref{chp_2_lemma_33_eq_2}, we deduce
		\begin{align*}
			\delta_{k+1} &\leq \left( \frac{\kappa_y}{\kappa_y + \tau}\right)^2\left(\norm{y^*_k-y_k}^2\right) + 2\norm{y^*_k-y_k}\norm{y^*_{k+1}-y^*_k}  + \norm{y^*_{k+1}-y^*_{k}}^2 \notag \\
			& \leq \left( \frac{\kappa_y}{\kappa_y + \tau}\right)^2 (1+\theta) \norm{y^*_k-y_k}^2 +\left( \frac{\kappa_y}{\kappa_y + \tau}\right)^2 \left(1+\frac{1}{\theta} \right)\norm{y^*_{k+1}-y^*_k}^2.
		\end{align*}
		for any $\theta>0$. Choosing $\theta^*=\frac{(2\kappa_y-\tau)(\kappa_y + \tau)^2}{2\kappa_y^3}-1$ that is strictly positive for $0< \tau \leq 1\leq\kappa_y$, we obtain $\delta_{k+1} \leq \left( 1 - \frac{\tau}{2\kappa_y} \right) \delta_k + \frac{\kappa_y}{\tau} \norm{ y_{k+1}^* - y_k^* }^2$, where we used the fact that $\left( \frac{\kappa_y}{\kappa_y + \tau}\right)^2 \left(1+\frac{1}{\theta^*} \right)
		=\frac{\kappa_y^2(2\kappa_y-\tau)}{\tau(3\kappa_y^2-\tau^2)}  \leq \frac{\kappa_y}{\tau}$ for $0<\tau\leq 1\leq\kappa_y$.
	\end{proof}

	\section{Proof of Lemma~\ref{Algo2_lemma_1} for the analysis of  \algotwo}\label{app:lem_alg2}
	
	\begin{proof}
		Since $\varphi$ is $\rho-$weakly convex (Assumption~\ref{assumption_2}(ii)), we have
		\begin{equation*}
			\varphi(x_{k+1}) \leq \varphi(x_k) + \langle \nabla \varphi(x_{k+1}), x_{k+1}-x_k \rangle + \frac{\rho}{2}  \norm{x_k -x_{k+1}}^2.
		\end{equation*}
		From first order optimality condition of \eqref{algo_2_update_x_sect_4} we have $x_{k+1}-x_k = -\eta_x \nabla_x \Phi(x_{k+1},y_k)$. Then
		\begin{equation}\label{Algo_2_lemma_1_eq0}
			\varphi(x_{k+1}) \leq \varphi(x_k) -\eta_x \langle \nabla \varphi(x_{k+1}), \nabla_x \Phi(x_{k+1},y_k) \rangle + \frac{\rho \eta^2_x}{2} \norm{\nabla_x \Phi(x_{k+1},y_k)}^2.
		\end{equation}
		\gt{ Using the polarization identity we get
			\begin{align}\label{Algo_2_lemma_1_eq1}
				-2 \langle \nabla \varphi(x_{k+1}), \nabla_x \Phi(x_{k+1}, y_k) \rangle = & - \norm{\nabla \varphi(x_{k+1})}^2 - \norm{\nabla_x \Phi(x_{k+1},y_k)}^2 \notag \\
				&+ \norm{\nabla \varphi(x_{k+1}) - \nabla_x \Phi(x_{k+1},y_k)}^2.
		\end{align}}
		Plugging \eqref{Algo_2_lemma_1_eq1} into \eqref{Algo_2_lemma_1_eq0} we get
		\begin{align} \label{Algo_2_lemma_1_eq2}
			\varphi(x_{k+1}) \leq &~\varphi(x_{k}) - \frac{\eta_x}{2} \norm{\nabla \varphi(x_{k+1})}^2 + \frac{\eta_x}{2} \norm{\nabla \varphi(x_{k+1}) - \nabla_x \Phi(x_{k+1},y_k)}^2 \notag \\
			&- \frac{\eta_x}{2} (1 - \eta_x \rho) \norm{\nabla_x \Phi(x_{k+1},y_k)}^2
		\end{align}
		Using Proposition \ref{Smoothness_of_the_max_function} we have 
		\begin{align} \label{Algo_2_lemma_1_eq3}
			\norm{\nabla \varphi(x_{k+1})-\nabla_x \Phi(x_{k+1},y_k)}^2 &=\norm{\nabla_x \Phi(x_{k+1},y^*(x_{k+1}))-\nabla_x \Phi(x_{k+1},y_k)}^2 \notag \\
			&\leq L_{xy}^2 \norm{y^*(x_{k+1})-y_k}^2. 
		\end{align}
		\gt{Plugging \eqref{Algo_2_lemma_1_eq3} into \eqref{Algo_2_lemma_1_eq2} and using that $\eta_x \rho < 1$ by assumption, we deduce 
			\begin{equation*}
				\varphi(x_{k+1}) \leq \varphi(x_k) -\frac{\eta_x}{2} \norm{\nabla \varphi(x_{k+1})}^2 + \frac{\eta_x}{2} L_{xy}^2 \norm{y^*(x_{k+1}) -y_k}^2.
		\end{equation*}}
		Summing over $k=0,\dots,N-1$ we deduce the result\footnote{\gt{Note that we do not drop the term $-\frac{\eta_x}{2} \norm{\nabla \varphi(x_{k+1})}^2$, since it will be used later to obtain an $\varepsilon$-stationary point.}}.
	\end{proof}

	\section{Additional Experimental Details}\label{sec::additionalExp}
	\paragraph{Computer architecture}
	\gt{We ran all the experiments on an HP Elite x360 1040 14-inch G11 2-in-1 Notebook PC with 16 GB of RAM and an Intel Core Ultra 5 125H CPU with 14 cores and a maximum frequency of 4.5 GHz.
		While the experiments related to imaging restoration were carried out using the PlaFRIM platform (see \url{https://www.plafrim.fr}).}
	
	\paragraph{Implementation of proximal operators}
	
	{\color{black}
	For the Toy example~\eqref{toy_example}, we have that $\Phi$ is $C^1$, $\rho=2$-weakly convex w.r.t. $x$, and $\mu=1$-strongly concave w.r.t. $y$.} Note that $\prox{\eta_x \Phi(\cdot,y_k)}(x_k)=\prox{\eta_x g}(x_k - \eta_x y_k)$ so following the closed formula for $\prox{\eta_x g}$ established in \cite[Exam. 9.1, pp. 21]{renaud2025moreauenvelope} we have that
		\begin{equation*}
			\prox{\eta_x g}(x_k - \eta_x y_k) = \begin{cases}
				(x_k - \eta_x y_k)/(1-2\eta_x) & \text{if} \quad \lvert x_k - \eta_x y_k \rvert \leq 0.5-\eta_x, \\
				(x_k - \eta_x y_k + 2\eta_x)/(1+2\eta_x) & \text{if} \quad x_k - \eta_x y_k \geq 0.5 - \eta_x, \\
				(x_k - \eta_x y_k - 2\eta_x)/(1+2\eta_x) & \text{otherwise}.
			\end{cases}
		\end{equation*}
	
	\gt{For the Dirac-GAN~\eqref{strongly_concave_dirac_gan} and RLR~\eqref{rlr_model} experiments, we implement $\prox{\Phi(\cdot,y_k)}$ using Newton's method, stopped after at most $500$ iterations or once the tolerance $\varepsilon = 10^{-5}$ is reached, whichever occurs first.
	}
	
	Finally, in the Plug-and-Play experiment~\eqref{min_max_formulation_pnp}, the proximal operator is parameterized via a neural network	$D_\sigma$. Here we take the \texttt{ProxDRUNet} denoiser \cite{Hurault2022ProximalDF} with pretrained weights\footnote{\url{https://github.com/samuro95/Prox-PnP}}.
	The test images are taken from the \texttt{Set3C} dataset available in the \texttt{DeepInverse} Python library \cite{deepinverse}. The observations $b$ are then generated using the \texttt{Downsampling} and \texttt{BlurFFT} operators implemented in \texttt{DeepInverse}.
	In practice, the minimized function of which $D_\sigma$ is the proximal operator is $\frac{L}{1-L}$-smooth, $\frac{L}{L+1}$-weakly convex with $L<1/2$, and can be evaluated together with $D_\sigma$. 
	In this setting, we must resort to using the step-size $\eta_x=1$, as $D_\sigma$ is the proximal operator for fixed step-size equal to one, and changing the step-size would mean retraining the network.
	
	Therefore, in this problem, the block-wise Lipschitz constants of Assumptions~\ref{assumption_1}-\ref{assumption_3} are $L_{xx}=\frac{L}{1-L}$, $L_{xy}=L_{yx}=\norm{A}_S$, $L_{yy}=\mu=\sigma^2/\lambda$, hence $\kappa_y=L_{yy}/\mu=1$. Note that the $y$ step-size for \algotwo becomes $0<\eta_y \leq \lambda/\sigma^2$. Since $L<1/2$, it follows that $\gt{\rho=\frac{L}{L+1}< \frac{1}{3}}$. Hence we have to guarantee $\eta_x=1<\min\left( \sigma^2/(\lambda (\gt{2 + 2\sqrt{2}})\norm{A}_S^2) ,3\right)$, which,  similarly to~\cite{Hurault2022ProximalDF}, restricts the value of the regularization parameter to $\lambda < \sigma^2/(\gt{ (2 + 2\sqrt{2}})\norm{A}_S^2)$. 
	

\paragraph{Acknowledgements} The experiments presented in this paper related to image restoration were
carried out using the PlaFRIM experimental testbed, supported by Inria, CNRS (LABRI and IMB), Université de Bordeaux, Bordeaux INP and Conseil Régional d’Aquitaine (see \url{https://www.plafrim.fr}).

		\FloatBarrier
		\bibliographystyle{plainnat}
		\bibliography{biblio.bib}
		
	\end{document}